\documentclass[a4paper]{amsart}

\usepackage{amsmath}
\usepackage{amssymb}
\usepackage{amsfonts}
\usepackage {graphicx}
\usepackage{amscd}
\usepackage[ruled,titlenumbered,vlined]{algorithm2e}
\usepackage[hang]{subfigure}

\newtheorem{theorem}{Theorem}
\newtheorem{lemma}[theorem]{Lemma}

\newtheorem{proposition}[theorem]{Proposition}

\newtheorem{definition}[theorem]{Definition}

\def\Claim#1.{\medbreak\ni{\bf Claim~#1.}\ }

\DeclareMathOperator{\Pois}{Pois}
\DeclareMathOperator{\Geom}{Geom}
\DeclareMathOperator{\Seq}{\textsc{Seq}}
\DeclareMathOperator{\MSet}{\textsc{MSet}}

\def\cM{\mathcal{M}}
\def\cA{\mathcal{A}}
\def\cB{\mathcal{B}}
\def\cC{\mathcal{C}}
\def\cP{\mathcal{P}}
\def\cZ{\mathcal{Z}}

\def\implies{\Longrightarrow}
\def\cp{\operatorname{copy}}

\def\ni{\noindent}

\newcommand{\cacher}[1]{}

\def\lar{\leftarrow}

\graphicspath{{Figures/}}

\begin{document}
\title{Random sampling of plane partitions}

\author[O. Bodini]{Olivier Bodini}
\address{O. Bodini: LIP6 - \'Equipe Spiral - 
104 av. du Pr\'esident Kennedy, 75016 Paris, France}  �
\email{bodini@calfor.lip6.fr}
\author[\'E. Fusy]{\'Eric Fusy}
\address{\'E. Fusy: LIX, \'Ecole Polytechnique, 91128 Palaiseau Cedex, France}
\email{fusy@lix.polytechnique.fr}

\author[C. Pivoteau]{Carine Pivoteau}
\address{C. Pivoteau: LIP6 - \'Equipe Spiral - 
104 av. du Pr\'esident Kennedy, 75016 Paris, France}  �
\email{carine.pivoteau@calfor.lip6.fr}

\date{\today}

\def\abstractname{abstract}

\begin{abstract} 
This article presents uniform random generators of plane partitions
according to the size (the number of cubes in the 3D interpretation). 
Combining a bijection of Pak with the method of Boltzmann sampling,    
we obtain random samplers that are slightly superlinear: the complexity is $O(n (\ln n)^3)$
in approximate-size sampling and $O(n^{4/3})$ in exact-size sampling
  (under a real-arithmetic computation model). 
To our knowledge, these are the first polynomial-time 
samplers for plane partitions according to the size (there exist polynomial-time  
samplers of another type, which draw plane partitions that fit inside a fixed bounding box). 
The same principles yield efficient 
 samplers for $(a\times b)$-boxed plane partitions (plane partitions with 
two dimensions bounded), and for skew plane partitions. The random samplers allow us
to perform simulations and observe limit shapes and frozen boundaries, which have been
analysed recently by Cerf and Kenyon for plane partitions, and by Okounkov and Reshetikhin  for
skew plane partitions. 
\end{abstract}

\maketitle

\section*{Introduction}

Plane partitions, originally introduced by A. Young~\cite{Yo01},
 constitute a natural generalisation of integer partitions in the plane, 
as they consist of a matrix of integers that are non-increasing in both 
dimensions (whereas an integer partition is an array of non-increasing integers). 
In addition, they also have a nice interpretation in 3D-space as a heap of 
cubes (see Figure~\ref{fig:pp_example}). 
 Plane partitions have
motivated a huge literature in numerous fields of mathematics 
\cite{BeEdKn,Ga81,Kn70,Mu,Wr31} and statistical
 physics \cite{MaNa,Ve96}, and have provided crucial insight for solving 
challenging problems in combinatorics~\cite{Ze96}, see~\cite{Bre99} for a 
detailed historical account. The problem of enumerating plane 
partitions was solved by MacMahon~\cite{Mc12}, 
who proved the beautiful formula 
\begin{equation}\label{eq:P}
P(x)=\prod_{r\geq 1}\frac1{(1-x^r)^{r}}
\end{equation}
 for the generating function. 
The simplicity of the formula asks for a combinatorial 
interpretation. A first direct bijective proof has been given by 
Krattenthaler~\cite{Kra99}. 
The principle is inspired by the seminal bijection of 
Novelli-Pak-Stoyanovskii~\cite{NoPaSt97} giving an interpretation of the 
hook-length formula. In~\cite{Kra99}, Krattenthaler also discusses as application 
of his 
bijection a polynomial-time algorithm for the random generation of 
plane partitions in a 
given $a\times b\times c$ box. Upon looking at the heap of cubes in the $(1,1,1)$ direction,
this task is equivalent to sampling tilings of a hexagon
of side lengths $(a,b,c,a,b,c)$ by rhombi; 
there also exist random samplers for such tilings, which rely either 
on ``coupling from the past principles''~\cite{PrGen} or on ``determinant algorithms''~\cite{Wil97}.
In contrast, we are interested here in sampling plane partitions uniformly at random
with respect to the \emph{size}, defined as the sum of the matrix entries.
For this purpose, we use another bijective interpretation of MacMahon's formula  
recently given by Pak~\cite{Pak02}.

Let us briefly mention the motivations for having a random sampler of  
plane partitions according to the size.
The size is a natural parameter, as it corresponds to the volume of the plane
partition (number of cubes) in the 3D interpretation.  Recently, several authors have studied  the
 statistical properties of plane partitions with respect to the size. In particular, 
 under fixed-size distribution, 
 Mutafchiev~\cite{Mu06} has shown a limit law for the maximal entry, and Cerf
and Kenyon~\cite{CeKe} have determined the asymptotic shape (the asymptotic shape
in the boxed framework ---hexagon tilings--- is due to Cohn, Larsen, and Propp~\cite{CoLaPr98}).
Even more recently, Okounkov and Reshetikhin, using a method based on
Schur processes,
have rediscovered the limit shape
of Cerf and Kenyon~\cite{okounkov-2003}. They have studied in a subsequent article~\cite{okounkov-2005} the local correlations and limit shapes for plane
partitions under a mixed model: the plane
partition is constrained to a
2-coordinate $a\times b$ box and is drawn under the Boltzmann model with respect to the size.
(We will also describe random samplers for this mixed model.) 
In addition, physicists have developed new models relying on 
plane partitions, giving rise to 
 a simplified version of the 3-dimensional models of lattice vesicles
 \cite{MaJa05}. Plane partitions are also related to the 
3-dimensional Ising model in the cubic lattice~\cite{CeKe}. 
In general, physicists are
 interested in checking experimentally or conjecturing some limit 
properties of these
 models, by generating very large random objects.

For this purpose, this paper introduces efficient samplers 
for plane partitions. Our approach combines methods from 
bijective combinatorics and symbolic combinatorics. Precisely, 
a minor reformulation of Pak' bijection 
 maps a multiset of integer pairs (the class is denoted by $\cM$) to a 
plane partition with the same size.
(Since the class $\cM$ has generating function $\prod_{r\geq 1}(1-x^r)^{-r}$,
this gives a direct proof of~\eqref{eq:P}.) 
Our aim here is to take advantage of this bijection for random sampling.  
Indeed Pak's bijection reduces the task of finding
a sampler for plane partitions to the task of finding a sampler for $\cM$.
As the class $\cM$ has an explicit simple combinatorial decomposition, it is
amenable to random sampling methods from symbolic combinatorics.
By now there is the \emph{recursive method}~\cite{NiWi78,FlZiVa94} 
based on the counting sequences 
and \emph{Boltzmann samplers} based on the generating functions, as introduced in~\cite{DuFlLoSc04} and further developed 
in~\cite{FlFuPi07}.
We adopt here the framework of 
 Boltzmann 
samplers, which tend to be more efficient as they avoid the costly precomputations
of coefficients required by the recursive method. 

As opposed to the recursive method ---which produces exact-size samplers--- the probability distribution in Boltzmann sampling is 
spread over the whole class; 
precisely an object of size $n$ has probability proportional to 
$x^n$, where $x$ is a fixed real parameter. In particular, as two
objects having the same size have equal probability, the 
probability distribution restricted to a given size $n$ is uniform. As we
are interested in generating very large plane partitions, the Boltzmann 
framework is suitable, due to the gain obtained by relaxing
the exact-size constraint.
 The articles~\cite{DuFlLoSc04}
and~\cite{FlFuPi07} provide a collection of rules for building a 
sampler for a
class admitting a decomposition involving classical constructions. Using
these rules, the decomposition of $\cM$ is readily translated into a 
Boltzmann sampler. This yields, via Pak's bijection, a Boltzmann 
sampler for plane partitions. In addition, as the size distribution
of plane partitions ---under the Boltzmann model--- has good
concentration properties, it is possible to ``tune''
 the parameter $x$ so as to draw objects of size around (or exactly at) 
a given target value $n$.  
With the parameter $x$ suitably tuned and a rejection loop targeted at the 
size,
 we obtain a \emph{quasi-linear time} 
approximate-size sampler for plane partitions: for any tolerance-ratio 
$\varepsilon\in(0,1)$, our sampler draws a plane partition of
size in $[n(1-\varepsilon),n(1+\varepsilon)]$ with an expected running
 $O(n (\ln n)^3)$. The same principles, with the rejection loop running
until a given size $n$ is attained, 
yields an exact-size sampler for plane partitions,
 with expected running time $O(n^{4/3})$. 
To our knowledge, our algorithm is the first exact-size sampler for
plane partitions with expected polynomial running time.  
This allows us
to generate objects of size up to $10^7$ in a few minutes on a PC. The same 
principles (i.e., Pak's bijection + Boltzmann samplers) yield
efficient Boltzmann samplers for $(a\times b)$-boxed plane partitions (plane
partitions whose non-zero entries lie in an ($a\times b$) rectangle), which are those considered
by Okounkov and Reshetikhin. 
We obtain for boxed plane partitions an approximate-size sampler with 
 expected running time $O_{a,b,\varepsilon}(1)$ and an exact-size 
sampler of expected running time $O_{a,b}(n)$, where $\varepsilon$ is a tolerance-ratio
on the size (for approximate-size sampling) and where $n$ is the target-size
(the dependency in $a,b,\varepsilon$ of the asymptotic constants in the big O's are stated precisely 
in Theorem~\ref{theo:box_plane}).    

Proving the correct complexity orders of the samplers is the major technical difficulty we have to
deal with. 
At first we have to analyse the expected running time of the Boltzmann 
sampler for the multiset-class $\cM$, as well as the size distribution on $\cM$ 
under the Boltzmann model. All this is done using the Mellin transform.  
Second, we study 
the complexity of Pak's bijection, which depends on a natural \emph{length-parameter}
of a plane partition, which is the maximum hook-length (abscissa+ordinate+1)
over all nonzero entries of the matrix. 
  Let us finally mention that, for the sake of simplicity, 
all complexity results are stated and proved with the $O$ notation (upper bound).
With little more care one could prove that all the stated complexity results
hold in fact with a $\Theta$ notation, i.e., an upper and a lower bound.

{\em\bf Outline of the paper.}
After some definitions in Section~\ref{Def} about combinatorial 
classes and plane
partitions, we present in Section~\ref{Pak} a slight reformulation (more algorithmic) of 
Pak's bijection. The bijection induces a combinatorial isomorphism between the
set of plane partitions and the class 
$\cM := \MSet(\cZ \times
\Seq(\cZ)^2)$. Section~\ref{Boltz} recalls basic principles of
Boltzmann sampling, in particular the sampling rules associated
to the constructions appearing in the specification of $\cM$.
 The Boltzmann sampler for $\cM$, as well as $\cM_{a,b}:=\MSet(\cZ\times \Seq_{<a}(\cZ) \times \Seq_{<b}(\cZ))$, is derived 
in Section~\ref{GammaP}, giving rise to Boltzmann samplers
for plane partitions and $(a\times b)$-boxed plane partitions. We explain then briefly how
the principles extend to obtain Boltzmann samplers for so-called skew plane partitions.
In Section~\ref{sec:target}, using suitable choices of the parameter $x$
in the Boltzmann samplers, we obtain efficient samplers for plane partitions
 targeted exactly or approximately
at a given size $n$  (precise statements are given in 
Theorems~\ref{theo:plane} and~\ref{theo:box_plane}).
The expected running times of the targeted samplers are then analysed in Section~\ref{sec:comp}.


\section{Definitions}\label{Def}
A \emph{combinatorial class} is a pair $(\mathcal{A},|.|)$ where 
$\mathcal{A}$ is a set and $|.|$ is a function from $\cA$ to $\mathbb{N}$, 
called the \emph{size function}, such that the
 number of
elements of any given size is finite. 
Using the size function, we can graduate $\mathcal{A}$ as
$\mathcal{A}=\bigcup_n\mathcal{A}_n$, where $\mathcal{A}_n$ is
the set of objects of $\mathcal{A}$ that have size $n$. In the sequel, we 
denote by $A_n$ the cardinality of
$\mathcal{A}_n$. To each combinatorial class $\cA$, we associate the 
\emph{generating function} $A(z)=\sum A_nz^n$.

Two combinatorial classes $(\mathcal{A},|.|_\mathcal{A})$ and
 $(\mathcal{B},|.|_\mathcal{B})$ are said to be
\emph{combinatorially isomorphic}, $\mathcal{A}\simeq \mathcal{B}$, if and 
only if there exists a one-to-one mapping from $\mathcal{A}$ to $\mathcal{B}$ 
that preserves the size. Let us notice that two classes $\cA$ and $\cB$ are 
isomorphic if and only if their generating functions are equal.

Here are some classical constructions on combinatorial classes that will be
 used in this paper. Notations and rules are summarized in 
Figure~\ref{tab:const} (a more general presentation can be found in 
\cite{FlaSe}): 
\begin{itemize}
\item[--]$\mathcal{E}$ and $\cZ$ are {\em atoms} of size $0$ and $1$.
\item[--]{\em Disjoint union} $\cA+\cB$:  the union of two 
 copies of $\cA$ and $\cB$ made disjoint. 
\item[--]{\em Cartesian product}  $\cA \times \cB$: the set of pairs 
$(\alpha,\beta)$ where $\alpha \in \cA$ and $\beta \in \cB$.
\end{itemize}
Given a class $\cA$ not containing empty atoms,
\begin{itemize} 
\item[--]Sequence: $\Seq(\cA)$ is 
the class of finite sequences of objects of $\cA$.
\item[--]Multiset: $\MSet(\cA)$ is the class of finite
 sets of objects of $\cA$, 
with repetitions allowed. 
\end{itemize}
In all these constructions, the size of an object in the composed class is
naturally defined as the sum of the sizes of the components
(e.g., the size of a sequence $\gamma_1,\ldots,\gamma_k$ 
is $|\gamma_1|+\cdots+|\gamma_k|$).  
Observe that, in a multiset $\mu\in\MSet(\cA)$, each element $\alpha\in\cA$ has a {\em multiplicity} $c_{\alpha}\geq 0$. Hence, if $\cA$ is a finite set,
\begin{equation}
\label{eq:prod}
\MSet(\cA)\simeq \prod_{\alpha\in\cA}\Seq(\{\alpha\}).
\end{equation}

\begin{figure}[htbp]\small
\begin{center}
\begin{tabular*}{\linewidth}{lll}
\hline
Class & Generating function & Definition\\
\hline
\hline
$\mathcal C=\mathcal{E}$ &   $C(z)=1$ & neutral object of size 0
\\
 $\mathcal C=\mathcal Z$ & $C(z)=z$ & atom of size 1
\\
\hline
$\mathcal C=\mathcal A+\mathcal B$ & $C(z)=A(z)+ B(z)$ &  disjoint union
\\
$\mathcal C=\mathcal A\times\mathcal B$ & $C(z)=A(z)\times B(z)$ &  cartesian
 product
\\
$\mathcal C=\Seq(\mathcal A)$ & $C(z)=(1-A(z))^{-1}$ & $\mathcal{E} + \mathcal{A} + \mathcal{A} \times \mathcal{A} + \mathcal{A} \times \mathcal{A} \times \mathcal{A} + ...$
\\
$\mathcal C=\MSet(\mathcal A)$ & $C(z)=\exp(\sum A(z^k)/k)$ & a multiset 
of elements of $\cA$
\\
\hline
\end{tabular*}
\end{center}
\caption{\label{tab:const} Some constructions on combinatorial classes.}
\end{figure}

A {\em plane partition} (Figure~\ref{fig:pp_example}) of $n$ is a 
two-dimensional array of non-negative integers $(a_{i,j})_{\mathbb{N}^2}$ that 
are non-increasing
both from left to right and bottom to top and that add up to $n$. 
In other words,
\begin{equation}
 \forall(i,j) \in \mathbb{N}^2~~~a_{i,j} \geq a_{i,j+1},~~~~~~a_{i,j} \geq a_{i+1,j}~~~~\mbox{ and }~~~\sum\limits_{i,j} {a_{i,j} }=n.
\end{equation}
We denote by $\mathcal{P}$ the combinatorial class of plane partitions, 
endowed with the size function $ \left|
{\left( {a_{i,j} } \right)_{\mathbb{N}^2}} \right| = \sum_{i,j} {a_{i,j} }$. 
Plane partitions have a natural representation in 3D-space as a heap 
of cubes with non-increasing height in the direction of the $x$-axis and 
$y$-axis, see Figure~\ref{fig:pp_example}. 
Observe that the size of the plane partition exactly corresponds to
the number of cubes in the 3D-representation.

The {\em bounding rectangle} of a plane partition $(a_{i,j})_{\mathbb{N}^2}$
is the smallest double range $R=[0..\ell-1]\times[0..w-1]$ such that $a_{i,j}=0$ 
for all index pairs $(i,j)$ outside of~$R$.
An~{\em $(a\times b)$-boxed plane partition} is a plane partition whose
 bounding rectangle is at most~$a\times b$.
Equivalently, $a_{i,j}$ is null for any $(i,j)$ such that $i\geq a$ or $j\geq b$.
 We denote by $\cP_{a,b}$ the class of $(a\times b)$-boxed plane 
partitions.
\begin{figure}[htbp]
\centerline{\includegraphics[scale=0.36]{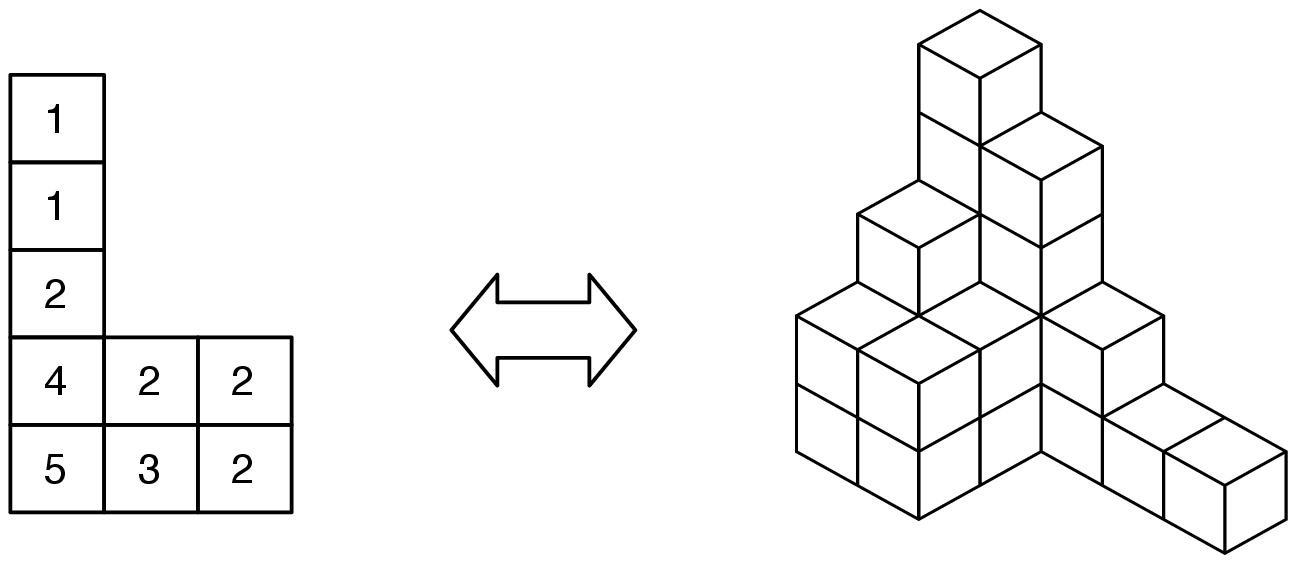}}
\caption{Plane partition of size 22 and its 3D representation.}\label{fig:pp_example}
\end{figure}
Define the two combinatorial classes $\cM$ and $\cM_{a,b}$ as follows, where $\Seq_{<d}(\mathcal{A})$ denotes the class of
 sequences of at most $d-1$ elements of $\mathcal{A}$.
\begin{eqnarray}
\cM & := &\MSet(\cZ \times \Seq(\cZ)^2)\label{eq:M}\\
\cM_{a,b} & := & \MSet(\cZ\times \Seq_{<a}(\cZ) \times \Seq_{<b}(\cZ)).\label{eq:Mpq}
\end{eqnarray}

Classically, $\Seq(\cZ)$ is identified with the class of nonnegative integers, 
so that we can specify $\cM$ with the following simplified notation,
\begin{equation} 
	\cM \simeq \MSet(\mathcal{Z}\times \mathbb{N}^2).\label{eq:seq_to_int}
\end{equation}
In the next section we explain how Pak's bijection yields an 
explicit combinatorial isomorphism between $\mathcal{P}$ and $\mathcal{M}$. 
For this purpose,
 we introduce some more terminology. The {\em diagram} $D$ of an 
element $M\in\cM$ is a two-dimensional array  $(m_{i,j})_{\mathbb{N}^2}$ 
(with $(0,0)$ at the bottom left) where  $m_{i,j}$ is the multiplicity of 
$(\cZ,i,j)$ in $M$ (see the first two pictures of 
Figure~\ref{fig:ex_pak}). The {\em size} of $D$ is defined as 
$|D|=\sum_{i,j} m_{i,j}(i+j+1)$, so that it corresponds to the size
of the multiset in $\cM$.
The \emph{bounding rectangle} of $M$ is defined similarly 
as for plane partitions: it is the smallest double range 
$R=[0..\ell-1]\times [0..w-1]$ such that all entries of $D$ outside of $R$ 
are zero. The integers
$\ell$ and $w$ are respectively called the \emph{length} and the \emph{width} 
of $D$. Note that, for fixed integers $a$ and $b$, 
the diagrams of elements in $\cM_{a,b}$ are constrained
to have their bounding rectangle $\subseteq[0..a-1]\times[0..b-1]$. 
Therefore $\cM_{a,b}$
is called the class of \emph{$(a\times b)$-boxed multisets}. 



\section{Pak's bijection}\label{Pak}

In~\cite{Pak02}, Pak presents a bijection between plane partitions 
bounded in a shape $\mu$ ($\mu$ being a Ferrers diagram) and fillings 
of the entries of $\mu$
with nonnegative integers. We reformulate this bijection as an 
algorithm, Algorithm~\ref{algo:pak} below, that realises explicitly the combinatorial isomorphism $\cM\simeq\cP$, 
see Figure~\ref{fig:ex_pak} for an example. 
\bigskip

\begin{algorithm}[H]\small
\caption{From the diagram of a multiset to a plane partition}\label{algo:pak}
\SetKwInOut{Input}{Input}
\SetKwInOut{Output}{Output}
\SetKw{KwDownto}{downto}

\Input{The diagram $D$ of a multiset in $\mathcal{M}$.}
\Output{a plane partition.}

Let $\ell$ be the length and $w$ be the width of $D$.\\
\For{$i\lar \ell-1$ \KwDownto $0$}{
\For{$j\lar w-1$ \KwDownto $0$}{
$D[i,j]\leftarrow D[i,j]+\max(D[i+1,j]),D[i,j+1])$\;
\For{$c\lar 1$ \KwTo $\min(w-1-i,\ell-1-j)$}{
$x \leftarrow i+c$; $y \leftarrow j+c$\;
$D[x,y]{\tiny \leftarrow} \max(D[x{\tiny +}1,y],D[x,y{\tiny +}1]){\tiny +}\min(D[x{\tiny -}1,y],D[x,y{\tiny -}1]){\tiny -}D[x,y]$\; } }}
\Return{$D$}\;
\end{algorithm}

\begin{proposition}[\em Pak \cite{Pak02}]\label{prop:plane}
Algorithm~\ref{algo:pak} yields an explicit size-preserving bijection between
the class $\cM_{a,b}$ and the class of $(a\times b)$-boxed plane partitions.
In other words, the algorithm realises  the combinatorial 
isomorphism
\begin{equation}
\mathcal{P}_{a,b}  \simeq  \MSet(\cZ\times \Seq_{<a}(\cZ) \times \Seq_{<b}(\cZ)).
\end{equation}
\end{proposition}
\begin{proof}
See~\cite{Pak02}.
\end{proof}
\begin{figure}[htbp]
\centerline{\includegraphics[scale=0.64]{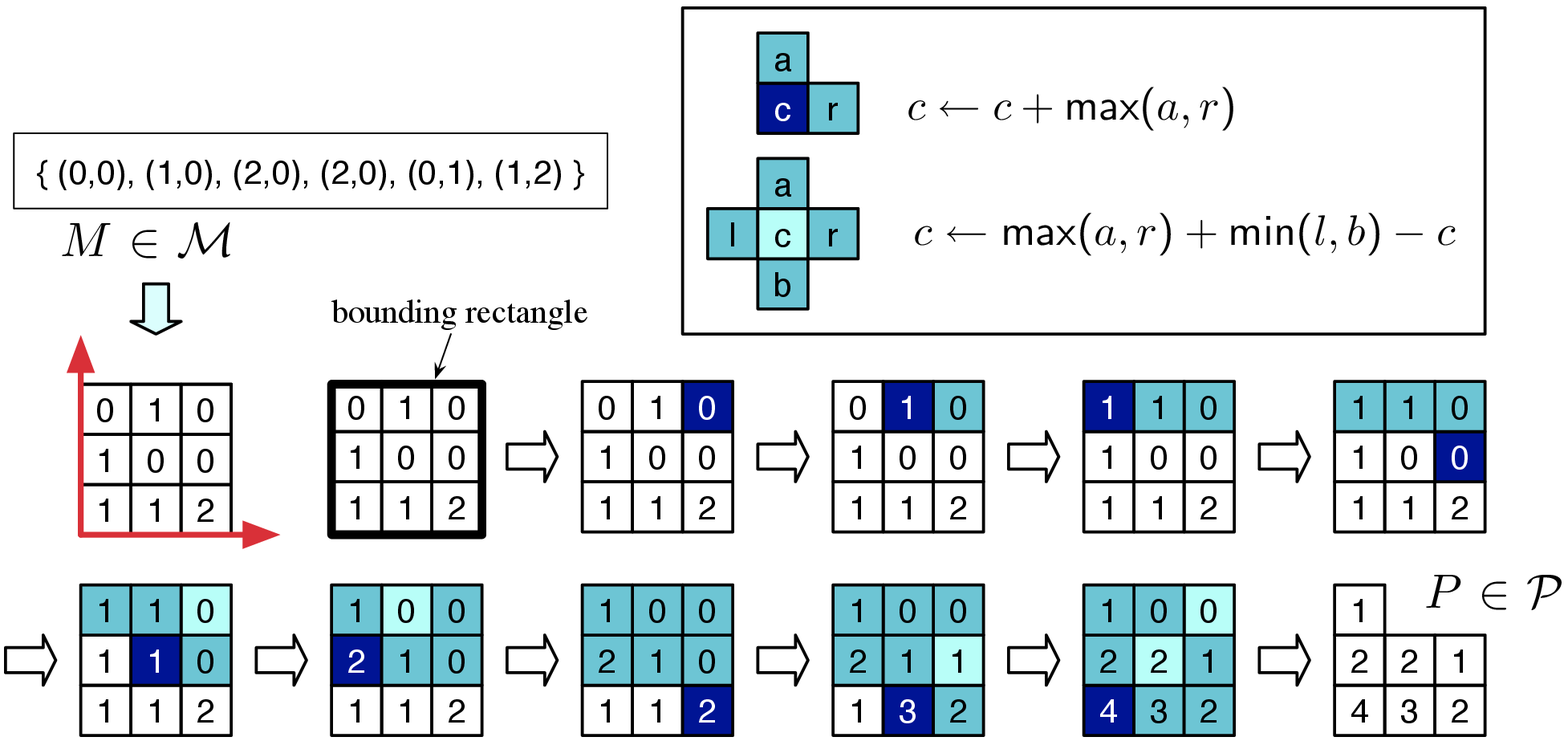}}
\caption{Pak's bijection on an example.}\label{fig:ex_pak}
\end{figure}

\begin{proposition}\label{prop:pp_mset}
Algorithm~\ref{algo:pak}  realises the combinatorial 
isomorphism
\begin{equation}
\mathcal{P}  \simeq  \MSet(\cZ\times \Seq(\cZ)^2).
\end{equation}
\end{proposition}
\begin{proof}
Take the limit $a\to\infty$ and $b\to\infty$ in Proposition~\ref{prop:plane}.
\end{proof}


\section{Boltzmann sampling}\label{Boltz}

This section recalls basic principles of approximate-size sampling 
under Boltzmann model (\cite{DuFlLoSc04,FlFuPi07}).

\begin{definition}[Boltzmann model]
	Let $\mathcal{C}$ be a combinatorial class and $C(x):=\sum_{\gamma\in\mathcal{C}}x^{|\gamma|}$ its generating function. 
Given a coherent positive real value of $x$, {\em i.e.}, chosen within the disk of
 convergence of $C(x)$), the  {\em Boltzmann model} of parameter $x$ 
assigns to any element $\gamma \in \mathcal{C}$ the following probability, 
\[\mathbb{P}_x(\gamma)=\frac{x^{|\gamma|}}{C(x)}.\]
\end{definition}

A {\em Boltzmann sampler} $\Gamma C(x)$ for 
$\mathcal{C}$ is an algorithm that produces objects of $\mathcal{C}$
at random under the Boltzmann model. As elements of the same size have
the same weight, the probability induced by 
a Boltzmann sampler on any given size $n$ is uniform. 
The size of the output is a 
random variable $N_x$ satisfying
\[ \mathbb{P}(N_x=n)=\frac{C_nx^n}{C(x)}.\]

Figure~\ref{fig:distrib} shows this probability distribution for plane 
partitions. When a target size $n$ has to be achieved, the idea is to tune 
the parameter $x$ so that $\mathbb{E}(N_x)=n$ (see Section~\ref{sec:comp}).

\begin{figure}[htbp]
\includegraphics[height=4cm]{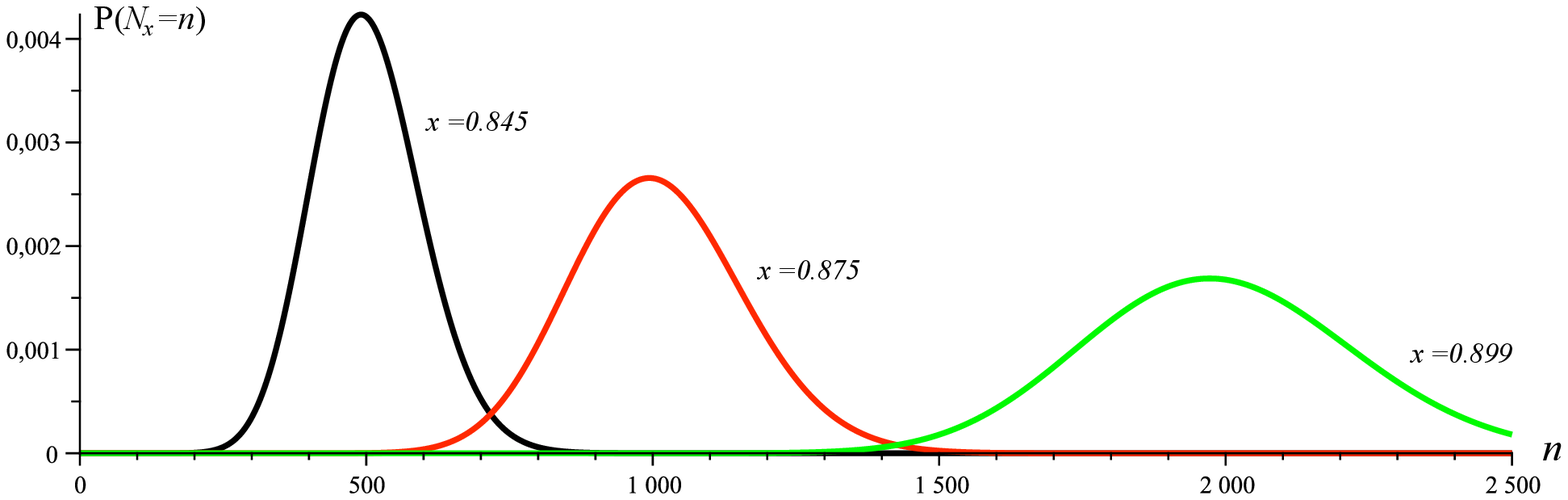}
\caption{\label{fig:distrib}\small Probability distribution of sizes for 
plane partitions under Boltzmann model, with different values of the 
parameter $x$.}
\end{figure}

Figure~\ref{tab:gen} briefly summarizes how to obtain samplers for the constructions used in the specification of $\mathcal{M}$ (see details 
in \cite{FlFuPi07}); the rules  can be combined 
to build a generator for any class specified with these constructions, in particular the class $\cM$.  
The sampling rules make use of simple auxiliary generators:
$\Geom(p)$  generates integers under the geometric law 
$\mathbb{P}(k)=p^k(1-p)$, $\Pois(\lambda)$ generates
integers under the Poisson law $\mathbb{P}(k)=e^{-\lambda}\frac{\lambda^k}{k!}$,
and $\Pois_{\geq 1}(\lambda)$ generates integers under the 
positive Poisson law $\mathbb{P}(k)=(e^{\lambda}-1)^{-1}\frac{\lambda^k}{k!}$ for $k>0$.
 Such generators are classically realised by simple iterative loops, the  complexity of generating an
integer $k$ being
 $O(k)$, see~\cite{DuFlLoSc04} for a discussion.

\begin{figure}\small
\begin{center}
\begin{tabular*}{\linewidth}{l|l}
\hline
$\mathcal C=\mathcal A \times \mathcal B$ &
\begin{minipage}{0.8\linewidth}$\Gamma C(x) := \langle \Gamma A(x), \Gamma B(x)\rangle$\end{minipage}\\
\hline 
$\mathcal C=\Seq(\mathcal{A})$ & $\Gamma C(x) :=  \left[\Geom(A(x)) \implies \Gamma A(x)\right]$\phantom{$\big|$}\\
& \begin{minipage}{0.8\linewidth} 
where $G\implies \Gamma Y$ means `` return $G$ independent calls to $\Gamma Y$''.	
\end{minipage}\\
\hline
$\mathcal C=\MSet(\mathcal{A})$ 
&  Define the probability distribution relative to $\cA$ and $x$: \phantom{$\big|$}\\
&	\begin{minipage}{0.8\linewidth} 
	\[
	\Pr(K\leq k)=\prod_{j>k}\exp\Big(-\frac{1}{j}A(x^j)\Big).
	\]
	Let $\textsc{Max\_Index}(A;x)$ be a  generator according to  this
	distribution.

	\smallskip
	\begin{tabbing}
	XXXXXX \= XX \= XX \= XX \= XX \kill
	$\Gamma C(x)$~:
	\> $\gamma\leftarrow \varnothing$; $k_0\leftarrow \textsc{Max\_Index}(A;x)$;\\
	\> {\bf if} $k_0\neq 0$ {\bf then}\\
	\>\> {\bf for} $j$ {\bf from} $1$ {\bf to} $k_0-1$ {\bf do}\\
	\>\>\> $p\leftarrow\Pois\left( \frac{A(x^j)}{j}\right)$;\\
	\>\>\> {\bf for} $i$ {\bf from} $1$  {\bf to} $p$ {\bf do} \\
	\>\>\>\> $\gamma \leftarrow \gamma, \cp(\Gamma A(x^j)\ j\ \rm{times})$\\
	\>\> $p\leftarrow\Pois_{\geq 1}\left( \frac{A(x^{k_0})}{k_0}\right)$;\\
	\>\> for $i$ {\bf from} $1$ {\bf to} $p$ {\bf do}\\
	\>\>\> $\gamma \leftarrow \gamma, \cp(\Gamma A(x^{k_0})\ k_0\ \rm{times})$\\
	\>{\bf return} $\gamma$.
	\end{tabbing}
\end{minipage}\\
\hline
\end{tabular*}
\end{center}
\caption{\label{tab:gen} Sampling rules associated to Boltzmann samplers 
for some combinatorial constructions.}
\end{figure}

\begin{proposition}[Flajolet et al.~\cite{FlFuPi07}]
Given two combinatorial classes $\cA$ and $\cB$ endowed with Boltzmann samplers 
$\Gamma A(x)$ and $\Gamma B(x)$, the sampler $\Gamma C(x)$, as defined in the 
first entry of Figure~\ref{tab:gen}, is a Boltzmann sampler for $\cA\times\cB$.
Given a class $\cA$ not containing the empty atom and endowed with a 
Boltzmann sampler $\Gamma A(x)$, the samplers $\Gamma C(x)$, as defined 
in the second and third entry\footnote{We hereby correct an omission in the definition of the sampler for $\MSet(\cA)$ given in~\cite{FlFuPi07}, namely the test $k_0\neq 0$.} of Figure~\ref{tab:gen}, are respectively 
Boltzmann 
samplers for $\Seq(\cA)$ and for $\MSet(\cA)$. 

From these sampling rules, a class $\cC$ 
recursively specified from atomic sets in terms of these constructions can be endowed with a 
Boltzmann sampler $\Gamma C(x)$. The complexity of generating an object $\gamma\in\cC$
 is $O(|\gamma|)$. 
\end{proposition}

\noindent{\textbf{Complexity model.}} 
Let us say a few words on the specific real-arithmetic complexity model used for  Boltzmann samplers. 
First, notice that the samplers given in Figure~\ref{tab:gen} draw integers according to distributions ($\Geom$, $\Pois$, {\sc Max\_Index}) that require the exact values of  the generating functions of the classes involved.
Hence, such generating functions should be evaluated.  
The complexity model we adopt, as already defined in~\cite{DuFlLoSc04},
relies on the {\em oracle} assumption. This assumption allows us 
to separate the combinatorial complexity of the sampler from the complexity of evaluating
the generating functions (there are already some results~\cite{PiSaSo07} and work in progress dedicated to the latter issue). 
Given any combinatorial class $\cC$ specified recursively
using the constructions of Figure~\ref{tab:gen}, and given a value $x>0$ within
the disk of convergence of $C(x)$, 
we assume that an \emph{oracle} 
 provides, at unit cost, the exact values at $x$ of the generating functions for all classes 
 intervening in the decomposition of $\cC$. 
In practice, we work with a fixed precision (e.g., $20$ digits) and precompute the values of generating functions used by the Boltzmann sampler.
Finally let us come back to the complexity of drawing an integer $K$ under a certain discrete
distribution
$$
\mathbb{P}(K=k)=p_k,
$$
where the constant $p_k$ are known (if $p_k$ involves exact values of generating functions, the oracle provides the values of $p_k$). 
As discussed in~\cite{DuFlLoSc04}, drawing an integer $K$ under this distribution
is done by a simple loop,


\begin{algorithm}[ht]\small
\caption{Drawing an integer $K$ under an arbitrary distribution}\label{algo:Dist1}
\SetKw{KwDownto}{downto}
$U:=$ uniform(0,1); $S:=0$; $k:=0$\;
\While{$U<S$}{$S:=S+p_k$; $k:=k+1$;}
\Return{$k$}\; 
\end{algorithm}

 
 Thus, the cost of drawing $K$ is of the order of the value $k$ that is finally assigned
 to $K$. An exception is the case of the geometric law, which is simpler. Indeed,
 to draw $K$ under $\Geom(x)$ (with $x\in[0,1]$), it is enough to set
 $K=\lfloor\ln(U)/\ln(x)\rfloor$, where $U$ is uniform in $(0,1)$.


Hence, the cost of drawing a geometric law is $O(1)$. 

\section{Samplers for plane partitions}\label{GammaP}

\subsection{Boltzmann sampler for plane partitions}

The explicit bijection between $\mathcal{M}$ and $\mathcal{P}$ allows us 
to design a simple Boltzmann sampler for plane partitions, made of two steps: 
(i) generate
a multiset in $\cM$ under the Boltzmann model, (ii) apply Algorithm~\ref{algo:pak} (Pak's bijection) to the
diagram of the multiset generated. 

\begin{algorithm}[ht]\small
\caption{$\Gamma M (x)$ [Boltzmann sampler for $\cM$]}\label{algo:genM}
\SetKw{KwDownto}{downto}
$M$ is the diagram of the multiset to be generated\\ 
$\forall (x,y)$  $M[x,y]\leftarrow 0$\;
$k_0\leftarrow \textsc{Max\_Index}(A;x)$, where $A(x)=x/(1-x)^2$\;
\If{$k_0\neq 0$}{
\For{$k \lar 1$ \KwTo $k_0-1$}{

$p\leftarrow \Pois(\frac{x^k}{k(1-x^k)^2})$\;
\For{$i \lar 1$ \KwTo $p$}{
$x \leftarrow \Geom(x^k)$; $y \leftarrow \Geom(x^k)$\;
$M[x,y] \leftarrow M[x,y]+k$\;
}}

$p\leftarrow \Pois_{\geq 1}(\frac{x^{k_0}}{k_0(1-x^{k_0})^2})$\;
\For{$i \leftarrow 1$ \KwTo $p$}{
$x \leftarrow \Geom(x^{k_0})$; $y \leftarrow \Geom(x^{k_0})$\;
$M[x,y] \leftarrow M[x,y]+k_0$\;
}}
\Return{$M$}\;
\end{algorithm}

\begin{lemma}
Given $0<x<1$, the generator $\Gamma M(x)$  ---as defined in 
Algorithm~\ref{algo:genM}---  is a Boltzmann sampler for $\cM$.
\end{lemma}
\begin{proof}
The specification of $\cM$, given in Equation~\eqref{eq:seq_to_int},
 is translated to a Boltzmann sampler using the rules of Figure~\ref{tab:gen}. 
The translation is carried out directly on the diagram of the multiset 
 (recall that the entry $(i,j)$ of the diagram 
corresponds to the multiplicity of $(\cZ,i,j)$ in the multiset).
\end{proof}

Since Pak's bijection preserves the size, Algorithm~\ref{algo:gammaP} is a Boltzmann sampler
for plane partitions.

\begin{algorithm}[ht]\label{algo:gammaP}\small
\caption{$\Gamma P (x)$, with $0<x<1$ [Boltzmann sampler for plane partitions]}\label{algo:genP}
\SetKw{KwDownto}{downto}
Compute $\mu\leftarrow\Gamma M(x)$\;
Apply Algorithm~\ref{algo:pak} (Pak's bijection) to $\mu$\;
{\bf return} $\mu$
\end{algorithm}

Figure~\ref{tab:bench} shows computation times\footnote{Computations have been performed on a Mac OS X Power PC G4 1,42GHz, with 1GB of RAM and 512 kB of cache.} of $\Gamma P(x)$ 
for sizes up to $10^7$: the first line gives the time of 
generation of the multiset ($\Gamma M(x)$) and the second line gives
the computation time of Pak's bijection. The sampler has been implemented in {\textsc Maple} and the bijection in OCaml. As we can see, the 
complexity is dominated by Pak's bijection for objects 
of large size. This is confirmed by the analysis to be given in Section~\ref{sec:comp}: 
the complexity of
drawing a multiset of size around $n$ is of order $n^{2/3}$, while
the expected running time of Pak's bijection applied to a random  
multiset of size $n$ is of order $n(\ln n)^3$. Figure~\ref{fig:pp_geante} 
shows two large plane partitions generated by $\Gamma P(x)$ for $x$ close to 1, $x=0.947$ and $x=0.9866$.

\begin{figure}[htbp]\small
\begin{center}
\begin{tabular*}{\linewidth}{llccccc}
\hline
approx. size && $10^{3^{\phantom{1}}}$ & $10^4$ & $10^5$ & $10^6$ & $10^7$
\\
\hline
generation && $\sim0.1$ sec. & $\sim0.5$ sec. & $\sim$ 2-3 sec. & $\sim10$ sec. & $\sim60$ sec.
\\
Pak's transform && $\sim0.1$ sec. & $\sim0.3$ sec. & $\sim2$ sec. & $\sim$ 20-30 sec. & $\sim$ 8-9 min
\\
\hline
\end{tabular*}
\end{center}
\caption{\label{tab:bench} Time per generation for different sizes of plane 
partitions.}
\end{figure}

\bigskip

\begin{figure}[htbp]
\begin{center}
\subfigure[A random plane partition of size 15,256, 
generated by $\Gamma P(0.947)$.]
          { 
            \includegraphics[scale=0.27]{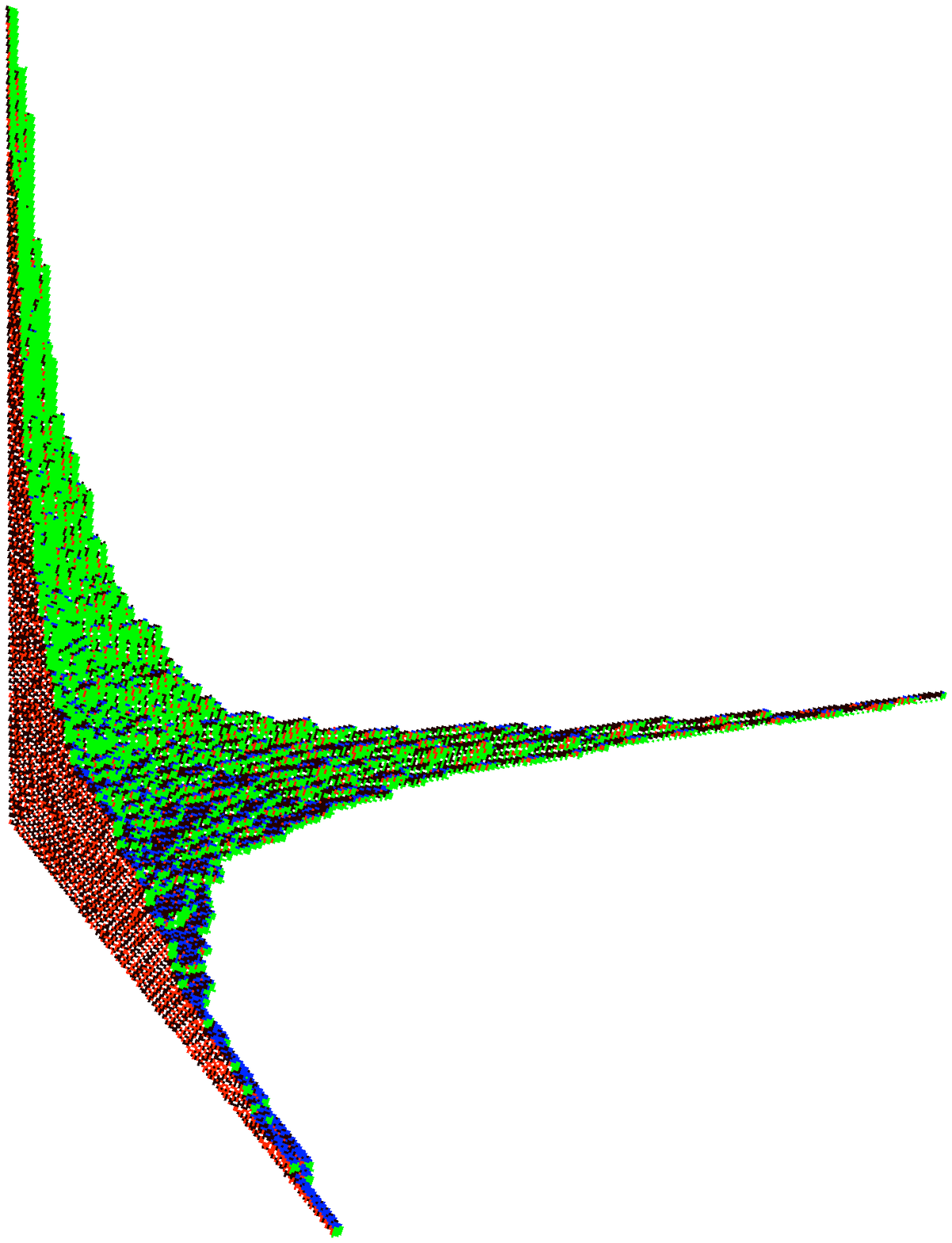}\label{fig:pp_geante_a}
          }
\hfill
\subfigure[ A random plane partition of size $1,005,749$ generated
by $\Gamma P(0.9866)$, seen from the direction $(1,1,1)$.]
          {
            \includegraphics[scale=0.27]{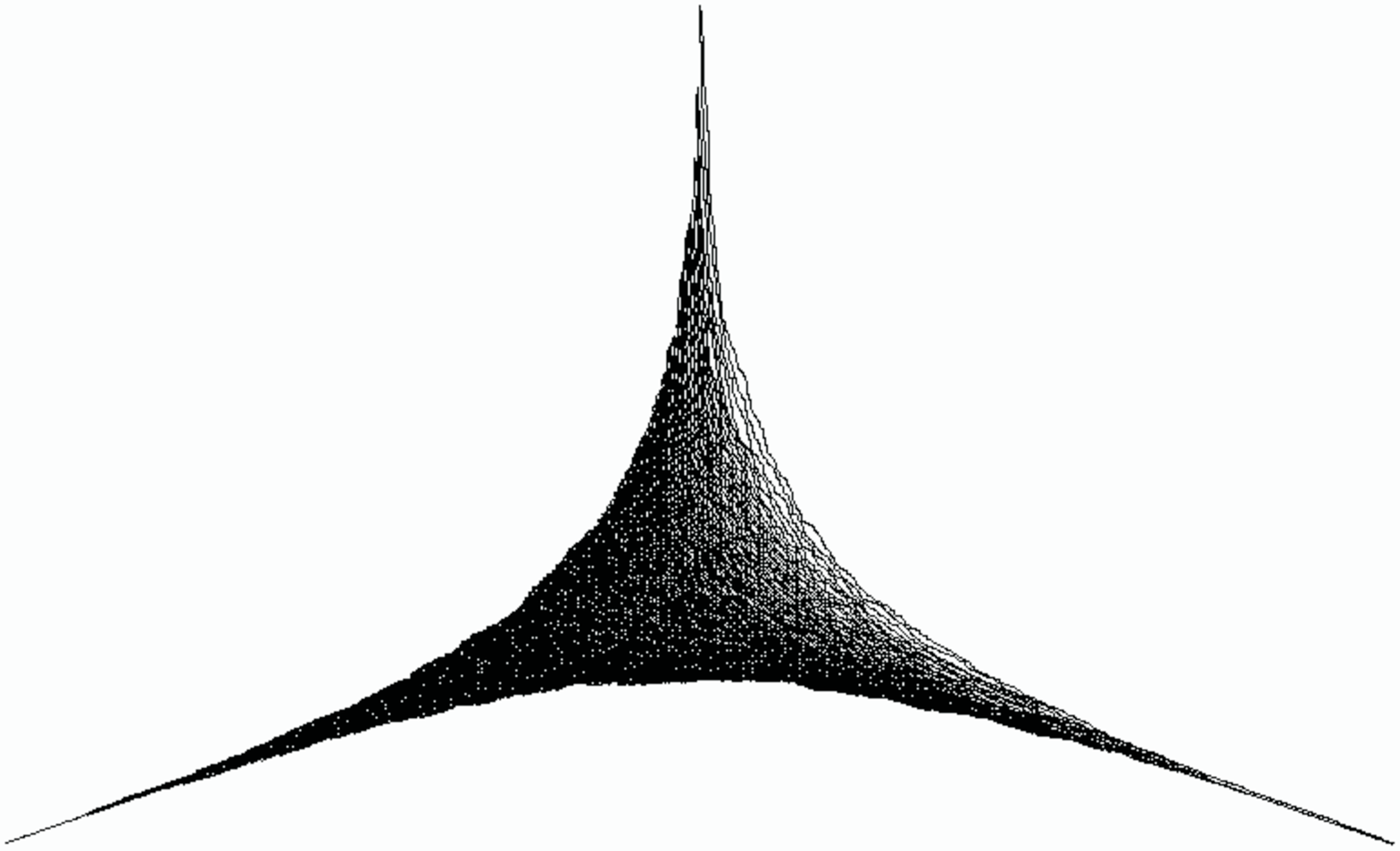}\label{fig:pp_geante_b}
          }
\end{center}
\caption{\label{fig:pp_geante}}
\end{figure}


\subsection{Boltzmann sampler for $(a\times b)$-boxed plane partitions}

According to the equivalence with the definition in terms of diagrams, an element of $\mathcal{M}_{a,b}$ is a multiset of pairs 
$(i,j)$, with $0\leq i<a$ and $0\leq j<b$, each element $(i,j)$ having size $(i+j+1)$. The set of such pairs 
being finite, Equation~(\ref{eq:prod}) yields
\begin{equation}
\label{eq:mpq}
\mathcal{M}_{a,b}=\prod_{\tiny\begin{tabular}{c} $^{0\leq i<a}$ \\ $^{0\leq j <b}$\end{tabular}} \Seq(\mathcal{Z}^{i+j+1})
\end{equation}

\begin{lemma}
Given $0<x<1$, the generator $\Gamma M_{a,b}(x)$ 
---as defined in Algorithm~\ref{algo:genMpq}--- 
is a Boltzmann sampler for $\cM_{a,b}$.
\end{lemma}
\begin{proof}
Translate the specification~(\ref{eq:mpq}) to a Boltzmann sampler for $\cM_{a,b}$ using the rules of Figure~\ref{tab:gen}.
\end{proof}

\begin{algorithm}[H]\small
\caption{$\Gamma M_{a,b} (x)$ [Boltzmann sampler for $\cM_{a,b}$]}
\label{algo:genMpq}
$M$ is the diagram of the multiset to be generated\\
\For{$i \leftarrow 0$ \KwTo $a-1$}{
\For{$j \leftarrow 0$ \KwTo $b-1$}{
$M[i,j] \leftarrow \Geom(x^{i+j+1})$\;
}}
\Return{$M$}\;
\end{algorithm}

Again, since Pak's bijection preserves the size, the following generator is a Boltzmann sampler
for $(a\times b)$-boxed plane partitions.

\begin{algorithm}[ht]\small
\caption{$\Gamma P_{a,b}(x)$, with $0<x<1$ [Boltzmann sampler for boxed plane partitions]}\label{algo:genPpq}
\SetKw{KwDownto}{downto}
Compute $\mu\leftarrow\Gamma M_{a,b}(x)$\;
Apply Algorithm~\ref{algo:pak} (Pak's bijection) to $\mu$\;
{\bf return} $\mu$
\end{algorithm}

\subsection{Extension to skew plane partitions}
We consider here a natural generalisation of $(a\times b)$-boxed plane partitions, called $(a\times b)$-boxed skew plane partitions. 
A $(a\times b)$-boxed \emph{skew plane partition} is given by an \emph{index-domain} $D\subset[0..a-1]\times[0..b-1]$ such that $D$ is obtained from $[0..a-1]\times[0..b-1]$
by removing rectangles of the form $[0..a'-1]\times[0..b'-1]$, with $a'\leq a$
and $b'\leq b$. Each truncation by a smaller rectangle makes an \emph{outer corner}
appear in the index domain (e.g., the partition of Figure~\ref{fig:skew_def} has 2 outer corners).
\begin{figure}[htbp]
\centerline{\includegraphics[scale=0.12]{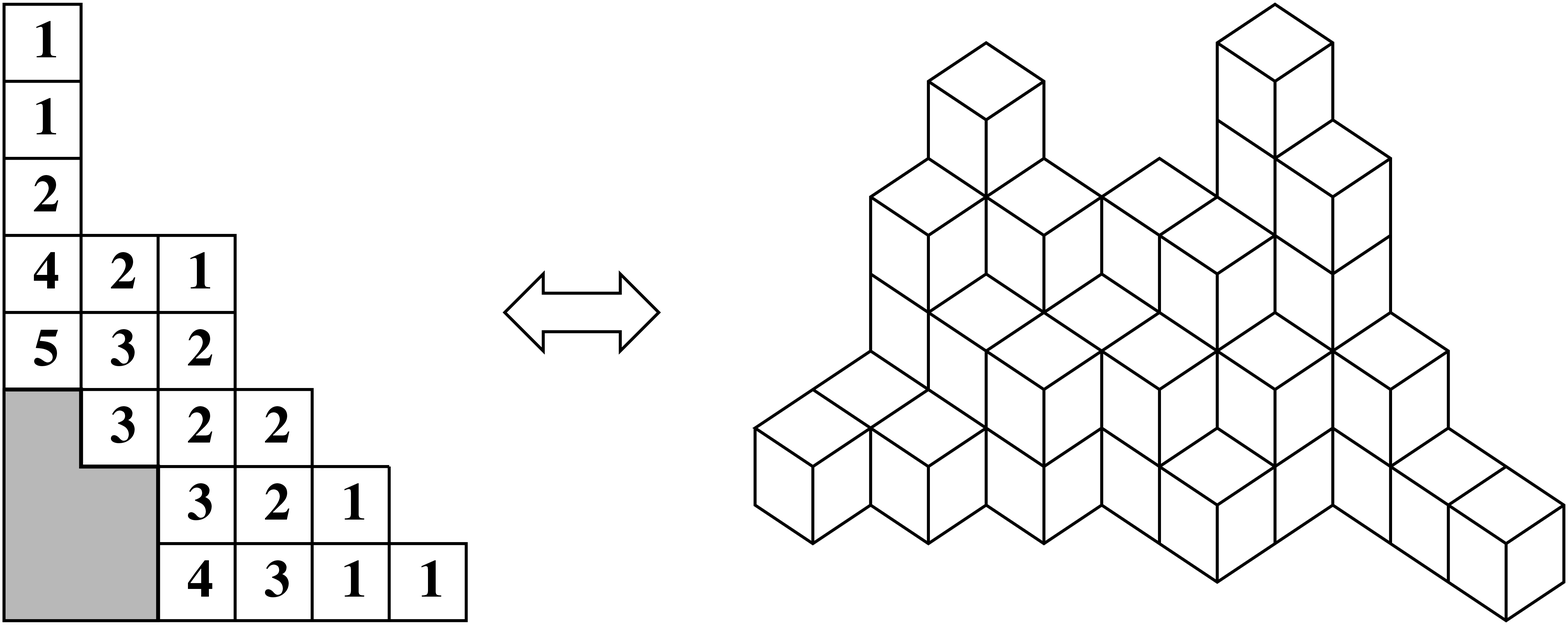}}
\caption{A skew plane partition of size 43.}\label{fig:skew_def}
\end{figure}
Let us denote by $\cP_D$ the class of all such partitions for a given domain $D$. For this new class of partitions, we need to define the hook-length of a pair $(i,j)$ in the domain $D$. Let $\ell(i)$ be the minimum abscissa such that $(\ell(i),j)\in D$ and $d(j)$ the minimum ordinate such that $(i,d(j))\in D$. The hook-length of $(i,j)$ in $D$ is then $h(i,j)=(i-\ell(i))+(j-d(j))+1$, which is exactly $i+j+1$ when $D$ is $[0..a-1]\times[0..b-1]$.
In~\cite{Pak02}, Pak's bijection is most generally described for skew plane partitions, which leads to the following combinatorial isomorphism:
$$\cP_D\simeq \prod_{(i,j)\in D}\Seq(\cZ^{h(i,j)})$$
The Boltzmann sampler  for $(a\times b)$-boxed plane partitions extends directly to a Boltzmann sampler for skew plane partitions as follows: 
to sample a diagram, draw the value at each point $(i,j)$ in $D$ according to a geometric law of parameter $x^{h(i,j)}$; then apply Algorithm~\ref{algo:pak} (Pak's bijection) to the multiset generated, with the difference that the domain scanned by $(i,j)$ is $D$.  

 Okounkov and Reshetikhin~\cite{okounkov-2005} 
 have studied the limit shape of a skew plane 
 partition under the Boltzmann distribution, 
 with the Boltzmann parameter $x$ tending to $1$.
  If the lengths of the rectangles are of order $(1-x)^{-1}$, some
 interesting phenomena are to be observed regarding 
 the typical shape of a random skew plane partition.
 Using a technique based on Schur processes, the authors of~\cite{okounkov-2005} provide a precise analysis 
 of these phenomena. They prove that the (rescaled) limit shape of a skew plane partition has a 
 frozen boundary that satisfies explicit equations, 
 and they classify the non-smooth points of the boundary as  
 \emph{turning points} and \emph{cusps}. 
 Turning points always appear, even for a boxed domain $(a\times b)$; they correspond to 
 points of tangency of the frozen boundary with the delimiting 3D-box 
 $\{(x,y,z)\in\mathbb{R}_+^3, s.t.\ (x,y)\in D\}$. 
 If the index domain has outer corners, some cusp points possibly appear at each of the
 outer corners.

Our random sampler for skew plane partitions makes it possible to perform simulations and observe these asymptotic phenomena. Figure~\ref{fig:boxed_example} shows a $(100\times 100)$-boxed plane partition of 
size 999,400. 
A frozen boundary appears that meets the delimiting 3D-box in a tangential way (these points of tangency are precisely the turning points in the 
terminology of Okounkov and Reshetikhin). And Figure~\ref{fig:skew_example} shows
a skew plane partition of size $1,005,532$  on the index-domain $(100\times 100)\backslash(50\times 50)$, which has an outer corner at $(50,50)$; accordingly
a cusp point appears on the boundary of the limit shape at the point above
 $(50,50)$. 

Note that the typical shape of a large unconstrained random plane partition, as shown in Figure~\ref{fig:pp_geante}, has different features: there are 
3 ``legs" ---one in each axis-direction--- whose  lengths tend to infinity even when the plane partition is rescaled to have unit volume (which 
essentially corresponds to rescaling by a factor $(1-x)^{-1}$ in each dimension).

\begin{figure}[htbp]
\begin{center}
\subfigure[A $(100\times 100)$-boxed plane partition of size 999,400 drawn under Boltzmann
distribution at $x=0.9931$.]
          { \hspace{0.4cm}
            \includegraphics[scale=0.4]{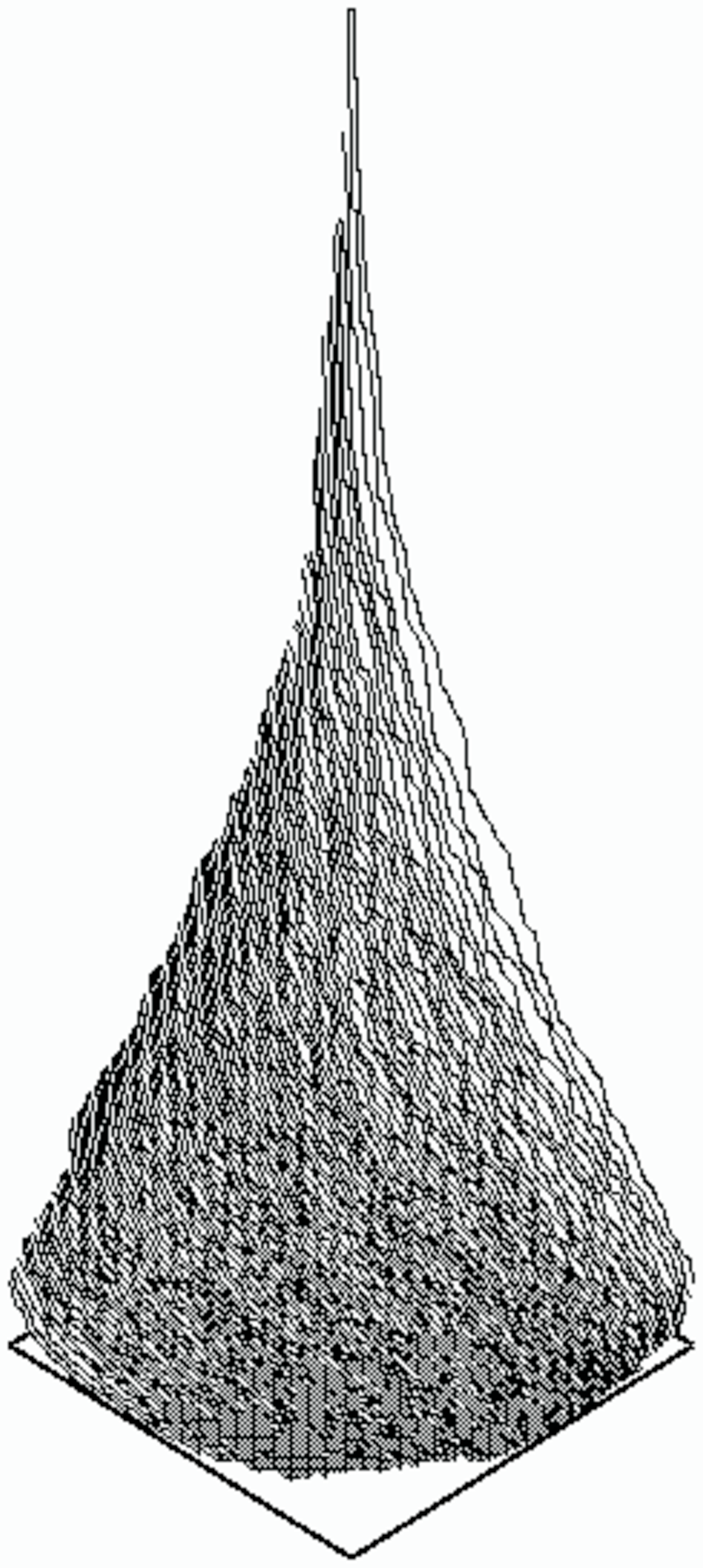}\label{fig:boxed_example}
            \hspace{0.4cm}
          }
\hspace{0.5cm}
\subfigure[ A $\lceil\hspace{-1.5mm}\lfloor0..99\rceil\hspace{-1.5mm}\rfloor \times \lceil\hspace{-1.5mm}\lfloor0..99\rceil\hspace{-1.5mm}\rfloor \backslash \lceil\hspace{-1.5mm}\lfloor0..49\rceil\hspace{-1.5mm}\rfloor  \times \lceil\hspace{-1.5mm}\lfloor0..49\rceil\hspace{-1.5mm}\rfloor$ skew plane partition of size $1,005,532$  drawn under Boltzmann distribution at $x=0.9942$.]
          { \hspace{0.6cm}
            \includegraphics[scale=0.4]{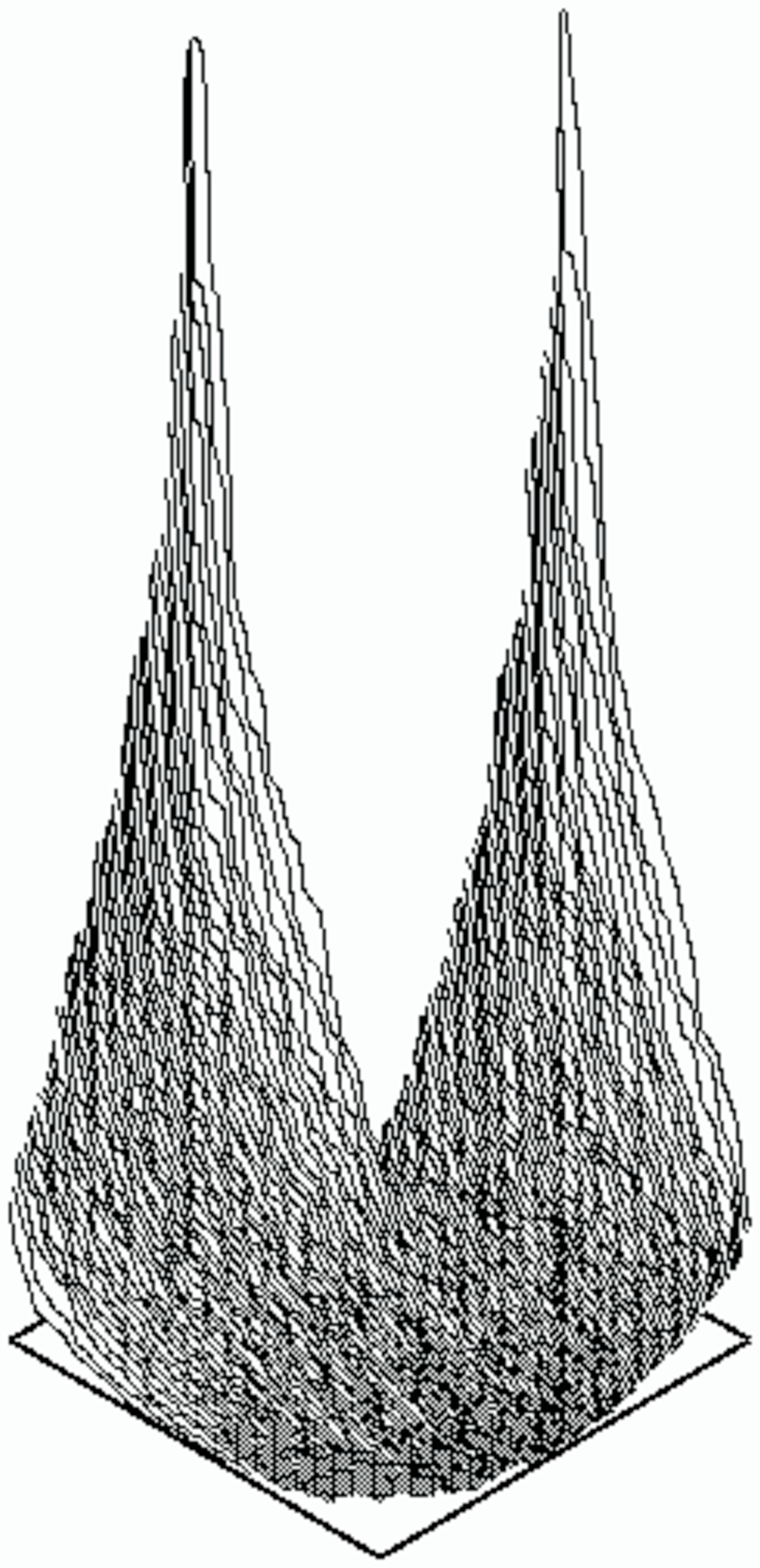}\label{fig:skew_example}
            \hspace{0.6cm}
          }
\end{center}\caption{}
\end{figure}

\subsection{Samplers targeted around a given size}\label{sec:target}
Given a class $\cC=\cup_n\cC_n$, 
an \emph{exact-size sampler} is a procedure that, for any given $n\geq 1$ 
(called the \emph{target-size}),
outputs an object of $\cC_n$ uniformly at random.
An \emph{approximate-size sampler} is a procedure that, for any given $n\geq 1$
and $\varepsilon\in(0,1)$ (called the \emph{tolerance-ratio}), 
outputs an object of $\cC$ of size in $[n(1-\varepsilon),n(1+\varepsilon)]$
and such that two objects of the same size have the same chance to be drawn (hence
the distribution induced on each size $k\in[n(1-\varepsilon),n(1+\varepsilon)]$ is  
uniform). 

Such procedures are easily obtained if  $\cC$ is endowed with a Boltzmann sampler 
$\Gamma C(x)$. 
Fix a suitable value of $x$ and repeat calling $\Gamma C(x)$ until the size is in the desired size-range $\Omega$;
 $\Omega=\{n\}$ for exact-size sampling and $\Omega=[n(1-\varepsilon),n(1+\varepsilon)]$
 for approximate-size sampling. The exact-size sampler and approximate-size sampler defined in this way
are denoted $\textsc{Sample}\cC(x;n)$ and $\textsc{Sample}\cC(x;n,\epsilon)$, respectively.
 
 In general one chooses $x$ so that the expected size $\Lambda C(x)$ of the output of $\Gamma C(x)$ ---which satisfies $\Lambda C(x)=xC'(x)/C(x)$ as proved in~\cite{DuFlLoSc04}--- is equal to $n$, or at least is asymptotically equal to $n$, that is, one looks for
 an exact or an approximate solution of the so-called \emph{target-size equation}:
 \begin{equation}
 x\frac{C'(x)}{C(x)}=n.
 \end{equation} 
 As we show in Section~\ref{sec:comp} 
 the expected size $\Lambda M(x)$ of the output of $\Gamma M(x)$ satisfies
 $$
 \Lambda M(x)\sim \frac{2\zeta(3)}{(1-x)^3},
 $$
 where $\zeta(.)$ is the Riemann zeta function, so a suitable value
 of the parameter $x$ to reach a target-size $n$ is $\xi_n\!=\!1-(2\zeta(3)/n)^{1/3}$,
 since $2\zeta(3)/(1\!-\!\xi_n)^3\!=\!n$;
 and we show in Section~\ref{sec:comp} 
 that the expected size $\Lambda M_{a,b}(x)$ 
 of the output of $\Gamma M_{a,b}(x)$ satisfies
 $$
 \Lambda M_{a,b}(x)\sim \frac{ab}{1-x},
 $$
so a suitable value
 of the parameter $x$ to reach target-size $n$ is $\xi_n^{a,b}=1-ab/n$.
  
%

\begin{lemma}[Targeted samplers for the multiset class $\cM$]\label{lem:targetM}
Define 
$$\xi_n:=1-(2\zeta(3)/n)^{1/3}.$$
Then, under the oracle assumption, the expected running time of $\textsc{Sample}\cM(\xi_n;n)$ 
is $O(n^{4/3})$; and, for fixed $\varepsilon\in(0,1)$,  
the expected running time of $\textsc{Sample}\cM(\xi_n;n,\varepsilon)$ is
$O(n^{2/3})$ as $n\to\infty$, the constant in the big 
$O$ being independent of $\varepsilon$~\footnote{Precisely, for each fixed $\epsilon$ there is $n_0(\epsilon)$
such that the running 
time is at most $c\ \!n^{2/3}$ for $n\geq n_0(\epsilon)$, where the constant $c$ is independent of $\epsilon$.}.
\end{lemma}

In view of stating the expected running times of the targeted samplers for
$(a\times b)$-boxed multisets, we define the following functions
$$
\phi(\alpha)=\frac{(\alpha/e)^{\alpha}}{\Gamma(\alpha)},\ \ \ \Phi(\alpha,\varepsilon):=\frac{(\alpha/e)^{\alpha}}{\Gamma(\alpha)}\int_{-\varepsilon}^{\varepsilon}
(1+s)^{\alpha-1}e^{-\alpha s}\mathrm{d}s.
$$

\begin{lemma}[Targeted samplers for $(a\times b)$-boxed multisets]
Define 
$$\xi_n^{a,b}:=1-ab/n.$$
Then, for $n\geq 1$, the expected running time of
 $\textsc{Sample}\cM_{a,b}(\xi_n^{a,b};n)$ is
 equivalent to $\phi(ab)/n$ as $n\to\infty$.
For fixed $(a,b,\varepsilon)$,  
 the expected running time
 of $\textsc{Sample}\cM_{a,b}(\xi_n^{a,b};n,\varepsilon)$ 
 converges  to the constant $ab/\Phi(ab,\varepsilon)$ as $n\to\infty$.  
 \end{lemma}

\begin{theorem}[Targeted samplers for plane partitions]
\label{theo:plane}
For $n\geq 1$, define the algorithm 
$\textsc{SamplePartitions}[n]$ as the procedure
that calls $\textsc{Sample}\cM(\xi_n;n)$ 
and applies Algorithm~\ref{algo:pak} (Pak's bijection)
 to the generated diagram.

Then $\textsc{SamplePartitions}[n]$
 is an exact-size sampler for 
plane partitions, of expected running time $O(n^{4/3})$. 

For $n\geq 1$ and $\varepsilon\in(0,1)$, define
 $\textsc{SamplePartitions}[n,\varepsilon]$ as the algorithm 
that calls $\textsc{Sample}\cM(\xi_n;n,\varepsilon)$ and
 applies Algorithm~\ref{algo:pak} to the generated diagram. 
Then $\textsc{SamplePartitions}[n,\varepsilon]$  is an 
approximate-size sampler for 
plane partitions, of  expected running time $O(n(\ln n)^3)$ as $n\to\infty$ (under fixed $\varepsilon$), 
the asymptotic constant in the big 
$O$ not depending on $\varepsilon$. 
\end{theorem}

The proofs of the expected running times announced 
in Theorem~\ref{theo:plane} and Theorem~\ref{theo:box_plane}
(given next) 
are delayed to Section~\ref{sec:comp}. 

In view of stating the expected running times of the targeted samplers for
$(a\times b)$-boxed plane partitions, we define the following function
$$
\psi(a,b):=\tfrac1{2}L\ell(\ell+1)-\tfrac1{6}(\ell^3-1),\ \ \mathrm{where}\ L=\mathrm{max}(a,b),\ \ell=\mathrm{min}(a,b).
$$

 \begin{theorem}[Targeted samplers for $(a\times b)$-boxed Plane Partitions]
\label{theo:box_plane}
For $n\geq 1$, define $\textsc{SamplePartitions}_{a,b}[n]$ as the 
algorithm that calls $\textsc{Sample}\cM_{a,b}(\xi_n^{a,b};n)$ 
and applies Algorithm~\ref{algo:pak} (Pak's bijection) to the 
generated diagram.

Then $\textsc{SamplePartitions}_{a,b}[n]$ 
 is an exact-size sampler for $(a\times b)$-boxed 
plane partitions, of expected running time equivalent to $\phi(ab)/n$ as $n\to\infty$. 

For $n\geq 1$ and $\varepsilon\in(0,1)$, define
 $\textsc{SamplePartitions}_{a,b}[n,\varepsilon]$ as the algorithm 
that calls $\textsc{Sample}\cM_{a,b}(\xi_n^{a,b};n,\varepsilon)$ and
 applies Algorithm~\ref{algo:pak} to the generated diagram.

Then $\textsc{SamplePartitions}_{a,b}[n,\varepsilon]$ 
is an 
approximate-size sampler for $(a\times b)$-boxed 
plane partitions, of  expected running time equivalent 
to the constant $\psi(a,b)+ab/\Phi(ab,\varepsilon)$ as $n\to\infty$ (under fixed $(a,b,\varepsilon)$).
\end{theorem}

Let us mention that the targeted samplers 
for $(a\times b)$-boxed plane partitions are easily
extended to the framework of $(a\times b)$-boxed \emph{skew} 
plane partitions. 
For a fixed admissible  
index-domain $D\subset[0..a-1]\times[0..b-1]$, 
the appropriate value to reach a target size $n$ exactly (or approximately) is
$\xi_n^{(D)}:=1-|D|/n$,  
 where $|D|$ is the cardinality of $D$.

Another 
important remark is that the value $\xi_n^{(D)}$ works well in the asymptotic regime,
that is, when $n>>|D|$. If not in the asymptotic regime (say one generates plane 
partitions of size $10,000$ constrained to a rectangular box $100\times 100$)
one has to consider the target-size equation more closely.
The generating function for multisets with support in $D$ is
$$
M_D(x)=\prod_{(i,j)\in D}\frac{1}{1-x^{i+j+1}}.
$$
Hence the target-size equation ---$xM_D'(x)/M_D(x)=n$--- is
$$
\sum_{(i,j)\in D}\frac{(i+j+1)x^{i+j+2}}{1-x^{i+j+1}}=n,
$$
which is to be solved exactly if $n$ is not in the asymptotic regime for the domain~$D$. 
(The solution  is 
asymptotically 
 $1-|D|/n$, but the rate of convergence is slow when $|D|$ is large.)



\section{Analysis of the complexity}\label{sec:comp}
This section is dedicated to proving the expected running times of the random samplers, as stated
in Theorem~\ref{theo:plane} and Theorem~\ref{theo:box_plane}. 
Since most of the difficulty is in proving Theorem~\ref{theo:plane} (unconstrained
plane partitions), the proof of Theorem~\ref{theo:box_plane} is only given
in the very last subsection (Section~\ref{sec:abboxed}) and is kept short. 

Recall that the random samplers 
consist of two steps: generate
a diagram under the Boltzmann model until the size is in the desired target-domain, 
and then apply Algorithm~\ref{algo:pak} (Pak's bijection) to the diagram so as 
to output a random plane partition. Accordingly, the complexity of generation is 
obtained by adding up the cost of generating a diagram and the cost of Pak's bijection. 

The expected costs of
generating diagrams under Boltzmann model are naturally expressed as 
certain infinite sums, which are 
best handled by the Mellin transform,  recalled next. 
On the other hand, Pak's bijection has complexity cubic  
in a certain parameter called the maximum hook-length, which is the maximal value 
of the hook-length
(abscissa+ordinate+1) over all nonzero entries. Therefore, 
 we need to find the asymptotic order of the maximum hook-length under
 the uniform distribution at size $n$.

\subsection{The Mellin  transform}
The Mellin  transform is a powerful technique to derive asymptotic estimates of 
expressions involving specific  infinite sums, (see~\cite{FlGoDu95} for 
a detailed survey), which occur  recurrently in the analysis of our samplers. 
Given a continuous function $f(t)$ defined on 
$\mathbb{R}^+$, the \emph{Mellin transform} of $f(t)$ is the function
\begin{equation}\label{eq:def_mellin}
f^*(s):=\int_0^{\infty}f(t)t^{s-1}\mathrm{d}t.
\end{equation}
For instance, the Euler Gamma function 
$\Gamma (s):=\int_0^{\infty}e^{-t}t^{s-1}\mathrm{d}t$ is the Mellin 
transform of $e^{-t}$. If $f(t)=O(t^{-a})$ as $t\to 0^+$ and 
$f(t)=O(t^{-b})$ as $t\to+\infty$, then $f^{*}(s)$ is an 
analytic function defined on the \emph{fundamental domain} 
$a<\mathrm{Re}(s)<b$. In addition, $f^*(s)$  is in most cases  continuable to a meromorphic 
function in the whole complex plane (for instance, $\Gamma(s)$ is continuable 
to a meromorphic function having its poles at negative integers). 
In a similar way as the Fourier transform, the Mellin transform is almost involutive, 
the function $f(t)$ being recovered from $f^{*}(s)$ using the 
\emph{inversion formula}
\begin{equation}\label{eq:inverse_mellin}
f(t)=\int_{c-i\infty}^{c+i\infty}f^{*}(s)t^{-s}\mathrm{d}s\ \ \ \ \mathrm{for\ any\ }c\in(a,b).
\end{equation}
From the inversion formula and the residue theorem, the asymptotic expansion of $f(t)$ as $t\to 0^-$ 
can be derived from the poles of $f^*(s)$ on the left of the fundamental 
domain, the rightmost such pole giving the dominant term of the asymptotic 
expansion. If $f^*(s)$ is decreasing very fast as 
$\mathrm{Im}(s)\to\infty$,  
 (which occurs in all the series to be analysed next, based on the 
fact that $\Gamma (s)$ is decaying fast and $\zeta(s)$ is of moderate growth as $\mathrm{Im}(s)\to\infty$), 
then there holds the following \emph{transfer rule}~\cite{FlGoDu95}: a pole of $f^*(s)$ of order $k\!+\!1$ ($k\geq 0$),  
$$
f^*(s)\mathop{\sim}_{s\to\alpha}\lambda_{\alpha}\frac{(-1)^kk!}{(s-\alpha)^{k+1}}
$$
yields a term 
$$\lambda_{\alpha}t^{-\alpha}(\ln t)^k$$
in the singular expansion of $f(t)$ around $0$. 
In particular, a simple pole $\lambda_{\alpha}/(s-\alpha)$ yields 
a term $\lambda_{\alpha}/t^{\alpha}$.
 
Another fundamental property of the Mellin transform is to \emph{factorize} sums of a certain form,
\begin{equation}
\label{eq:fac}
g(t)=\sum_{k\geq 1}a_kf(\mu_kt)\ \Rightarrow\ g^{*}(s)=\Big(\sum_{k\geq 1}a_k\mu_k^{-s}\Big)f^*(s).
\end{equation}

\subsection{Complexity of the Boltzmann samplers for multisets}
\label{sec:compP}
In this section, we analyse the complexity of the free Boltzmann sampler
 $\Gamma M(x)$ for the multiset class $\cM$,  
 not studying yet the rejection cost when targeting 
 at a certain size-domain. In general, 
 given a combinatorial class $\mathcal{C}$ for which an 
explicit Boltzmann sampler $\Gamma C(x)$ is designed, we write 
$\Lambda C(x)$ for the expected running time of a call to 
$\Gamma C(x)$. More generally, we use thereafter the letter $\Lambda$ as 
a prefix to denote the expected running time of a random generator.  
As we are interested in drawing large plane partitions, 
which requires to let $x$ tend to $1$, 
we analyse the asymptotic order of $\Lambda M(x)$  
as $x\to 1^-$. Recall that $\cA:=\cZ\star\Seq(\cZ)^2$, with generating function $A(x)=x/(1-x)^2$. 
By definition, the first step of the Boltzmann sampler $\Gamma M(x)$ 
is to draw an integer $K$ under the 
probability distribution
\begin{equation}\label{eq:probK}
\mathbb{P}(K\leq k)=\prod_{j>k}\exp\big(-\tfrac1{j}A(x^j)\big).
\end{equation}

 Under the oracle assumption discussed in Section~\ref{Boltz} and
 described in details in~\cite{DuFlLoSc04}, 
 the complexity of drawing $K$ is thus of the order
 of the value $k$ that is finally assigned to $K$. 
 Hence the expected running time of drawing $K$ is of the same order as 
 the expected value of $K$ under the above given distribution. 
 
 \begin{lemma}
 The expectation $\mathbb{E}_x(K)$ of $K$ under the distribution~\eqref{eq:probK}
 satisfies
 $$
 \mathbb{E}_x(K)=O((1-x)^{-1}\ln(1-x))\ \ \mathrm{as}\ x\to 1^-.
 $$
 \end{lemma}
 \begin{proof}
 Fix $x\in(0,1)$. Let $r=r(x)$ be the smallest integer such that $x^r< (1-x)/2$, i.e., 
 $r=\lfloor\ln((1-x)/2)/\ln(x)\rfloor+1$. Note that $x^i\leq 1/2$ for $i\geq r$.
  Hence, for $i\geq r$, $A(x^i)\leq 4x^i$. And, for $k\geq r$,
 $$
 \sum_{i>k}\frac1{i}A(x^i)\leq 4\sum_{i>k}\frac{x^i}{i}\leq\frac{4x^k}{1-x}\leq 2x^{k-r}.
 $$
 Hence, for $k\geq r$,  
 $$
 \mathbb{P}(K> k)=1-\mathbb{P}(K\leq k)\leq 1-\exp(-2x^{k-r})\leq 2x^{k-r}.
 $$
 We obtain thus
 $$
 \mathbb{E}_x(K)=\sum_{k\geq 0}\mathbb{P}(K>k)\leq r+2\sum_{k\geq r}x^{k-r}\leq r+\frac{2}{1-x},
 $$
 which concludes the proof since $r=r(x)$ is $O((1-x)^{-1}\ln(1-x))$ as $x\to 1^-$.
 \end{proof}
 
 Once the integer $K$ is drawn, the Boltzmann sampler $\Gamma M(x)$
draws Poisson laws and geometric laws (a call to $\Gamma A(x)$ consists
of two calls to geometric laws). 
Precisely, for each $i\geq 1$, the number of calls to $\Gamma A(x^i)$ 
follows a Poisson law $\Pois(A(x^i)/i)$. 
Since $\mathbb{E}(\Pois(\lambda))=\lambda$, the expected number of calls to 
$\Gamma A(x^i)$ is $A(x^i)/i$. In addition, each call to $\Gamma A$ takes constant time, since it consists
of two calls to geometric laws. Hence
\begin{equation}
\label{eq:multiset}
\Lambda M(x)=O\Big(\mathbb{E}_x(K)\Big)+O\Big(\sum_{i\geq 1}A(x^i)/i\Big)=O\Big(\frac{\ln(1-x)}{1-x}\Big)+O\Big(\sum_{i\geq 1}A(x^i)/i\Big).
\end{equation}

\begin{lemma}
\label{lem:expM}
The expected running time of the Boltzmann sampler $\Gamma M(x)$  satisfies  
$$
\Lambda M(x)=\mathop{O}_{x\to 1^-}\big((1-x)^{-2}\big).
$$
\end{lemma}
\begin{proof}
By Equation~\eqref{eq:multiset}, it is enough to show that 
$F(x):=\sum_{i\geq 1}A(x^i)/i$ is $O((1-x)^{-2})$ 
as $x\to 1^-$. This is a first instance where the Mellin transform can be
 successfully applied (more elementary approaches would work in this
 simple case). Define $L(t):=F(e^{-t})$.  Then
$$
L(t)=\sum_{r\geq 1}\frac{e^{-rt}}{r(1-e^{-rt})^2}=\sum_{r\geq 1}\frac{1}{r}f(rt),\ \mathrm{where\ } f(t):=\frac{e^{-t}}{(1-e^{-t})^2}.
$$
The factorization property of the Mellin transform, Equation~\eqref{eq:fac}, yields 
$$
L^*(s)=\big(\sum_{r\geq 1}\frac{1}{r}r^{-s}\big)f^*(s)=\zeta(s+1)f^*(s),
$$
where $\zeta(s):=\sum_{r\geq 1}r^{-s}$ is the Riemann zeta function.
Since $f(t)=\sum_{n\geq 1}ne^{-nt}$,  the factorization
 property yields $f^*(s)=\left(\sum_{n\geq 1}nn^{-s}\right)\Gamma(s)=\zeta(s-1)\Gamma(s)$. Thus, 
 $L^*(s)=\zeta(s+1)\zeta(s-1)\Gamma(s)$.
 It is easily checked that $L(t)=O(1/t^2)$ as $t\to 0^+$ and 
$L(t)=O(e^{-t})$ as
 $t\to\infty$, so that  the fundamental domain of $L^*(s)$ is $\mathrm{Re}(s)>2$.
 Hence, to determine the asymptotic behavior of $L(t)$ as $t\to 0^+$,
 we have to find the rightmost poles of $L^*(s)$ such that $\mathrm{Re}(s)\leq 2$.
 The function $\zeta(s)$ has a unique pole at $s=1$ with coefficient 1, 
and the function $\Gamma(s)$ has its poles at non-positive integers. 
Hence $L^*(s)$ has a simple pole at $s=2$, with coefficient $\zeta(3)$, and
 no other pole for $\mathrm{Re}(s)\geq 1$, so that the transfer rule of the Mellin 
transform yields
$$
L(t)=\frac{\zeta(3)}{t^2}+O\left(\frac{1}{t}\right)\ \ \
\mathrm{as\ }t\to 0^+.
$$
The change of variable $t=-\ln(x)$ yields
$$
F(x)=\frac{\zeta(3)}{(1-x)^2}+O\left(\frac{1}{1-x}\right)\ \
\ \mathrm{as\ }x\to 1^-.
$$
As a consequence,
 $F(x)=O((1-x)^{-2})$ as $x\to 1^-$.
\end{proof}

\subsection{Analysis of the size of a multiset in $\cM$ under the Boltzmann
model}\label{asppbp}
Given $0<x<1$, denote by $N_x$ the random variable giving the size of the 
output of $\Gamma M(x)$ (which is also the size of a plane partition
under the Boltzmann model at $x$); Figure~\ref{fig:distrib} shows plots of $N_x$ for several values of $x$. 
As the Boltzmann probability of an object of size 
$n$ is 
$x^n/M(x)$ , the expectation and variance of $N_x$ satisfy (see~\cite{DuFlLoSc04} for details):
$$
\mathbb{E}(N_x)=\sum_{n\geq 1}nM_n\frac{x^n}{M(x)}=x\frac{M'(x)}{M(x)},\ \  \mathbb{V}(N_x)=\sum_{n\geq 1}n^2M_n\frac{x^n}{M(x)}-\mathbb{E}(N_x)^2=x\frac{\mathrm{d}\mathbb{E}(N_x)}{\mathrm{d}x}.
$$

\begin{lemma}
\label{lemma:expvar}
The expectation and variance of the size of a multiset $\mu\in\cM$ drawn under  
Boltzmann model satisfy
$$
\mathbb{E}(N_x)=\frac{2\zeta(3)}{(1-x)^3}+\mathop{O}_{x\to 1^-}\left( \frac{1}{(1-x)^2}\right),\ \ \ \mathbb{V}(N_x)=\frac{6\zeta(3)}{(1-x)^4}+\mathop{O}_{x\to 1^-}\left(\frac{1}{(1-x)^3}\right).
$$
\end{lemma}
\begin{proof}
We use once again the Mellin transform to derive the asymptotic estimates. 
Observe that $M'(x)/M(x)$ is the logarithmic derivative of $M(x)$, hence 
the expression~\eqref{eq:P} of $M(x)=P(x)$ yields
$$
\mathbb{E}(N_x)=x\sum_{r\geq 1}r\frac{rx^{r-1}}{1-x^r}=\sum_{r\geq 1}r^2\frac{x^r}{1-x^r}.
$$
Define $L(t):=\mathbb{E}(N_{e^{-t}})$. Then 
$$L(t)=\sum_{r\geq 1}r^2\frac{e^{-rt}}{1-e^{-rt}}=\sum_{r\geq 1}r^2f(rt),$$
where $f(t):=e^{-t}/(1-e^{-t})=\sum_{n\geq 1}e^{-nt}$. Hence 
$$L^*(s)=\sum_{r\geq 1}(r^2r^{-s})f^*(s)=\zeta(s-2)f^*(s)=\zeta(s-2)\sum_{n\geq 1}n^{-s}\Gamma(s)=\zeta(s-2)\zeta(s)\Gamma(s).$$
The function $L^*(s)$ is defined on the fundamental domain $\mathrm{Re}(s)>3$. 
The rightmost pole such that $\mathrm{Re}(s)\leq 3$ is at $s=3$, 
where $L^*(s)\sim 2\zeta(3)/(s-3)$. As there are no other poles for $\mathrm{Re}(s)\geq 2$,
 the transfer rule yields
$$
L(t)=\frac{2\zeta(3)}{t^3}+O(t^{-2})\ \ \mathrm{as}\ t\to 0^+.
$$
Hence the change of variable $x=-\ln(t)$ gives
$$
\mathbb{E}(N_x)=\frac{2\zeta(3)}{(1-x)^3}+O\left(\frac{1}{(1-x)^2}\right)\ \ \ \ \mathrm{as}\ x\to 1^-.
$$
The variance is treated similarly,
$$
\mathbb{V}(N_x)=x\frac{\mathrm{d}\mathbb{E}(N_x)}{\mathrm{d}x}=\sum_{r\geq 1}r^3\frac{x^r}{(1-x^r)^2}.
$$
Hence the function $L(t):=\mathbb{V}(N_{e^{-t}})$ satisfies $L(t)=\sum_{r\geq 1}r^3g(rt)$, where $g(t)=e^{-t}/(1-e^{-t})^2=\sum_{n\geq 1}ne^{-nt}$.
Thus, $L^*(s)=\zeta(s-3)\zeta(s-1)\Gamma(s)$. The location of the poles of $L^*(s)$ and the transfer rule yields 
$$L(t)=\frac{6\zeta(3)}{t^4}+O(t^{-2})\ \ \mathrm{as}\ t\to 0^+,$$
giving
$$
\mathbb{V}(N_x)=\frac{6\zeta(3)}{(1-x)^4}+O\left(\frac{1}{(1-x)^3}\right)\ \ \ \ \mathrm{as}\ x\to 1^-.
$$
\end{proof}

\subsection{Complexity of the targeted samplers for multisets}\label{cts}
Recall that the targeted samplers for the multiset class $\cM$ repeat calling 
the Boltzmann sampler $\Gamma M(x)$  with a suitable value of $x$ 
until the size is in the target domain $\Omega$; $\Omega=\{n\}$ for exact-size
sampling and $\Omega=[n(1-\varepsilon),n(1+\varepsilon)]$ for approximate-size sampling. 

\begin{lemma}
\label{lem:prob}
For $n\geq 1$, let $\xi_n$ be the solution of $2\zeta(3)/(1-x)^3=n$, 
i.e.,
$$
\xi_n=1-(2\zeta(3)/n)^{1/3}.
$$
Define $\pi_n$ as the probability that the output of $\Gamma M(\xi_n)$ has 
size $n$. For any $\varepsilon\in(0,1)$, define $\pi_{n,\varepsilon}$ as the probability 
that the size of the output of $\Gamma M(\xi_n)$ is in the range 
$[n(1-\varepsilon),n(1+\varepsilon)]$. Then, $\pi_n\sim Cn^{-2/3}$ as $n\to\infty$,  
with $C\approx 0.1082$; and, for fixed $\varepsilon\in(0,1)$, 
$\pi_{n,\varepsilon}\to1$ as $n\to\infty$.
\end{lemma}
\begin{proof}
As $\xi_n$ is solution of $2\zeta(3)/(1-x)^3=n$, Lemma~\ref{lemma:expvar} 
ensures that $\mathbb{E}(N_{\xi_n})=n+O(n^{2/3})$ as $n\to \infty$, 
i.e., there exists $C>0$ such that $|\mathbb{E}(N_{\xi_n})-n|\leq Cn^{2/3}$. 
Hence
 Chebyshev's inequality gives, for any $\varepsilon\in(0,1)$,
\begin{eqnarray*}
1-\pi_{n,\varepsilon}&=&\mathbb{P}(|N_{\xi_n}-n|>\varepsilon n)\\
&\leq& \mathbb{P}\left(|N_{\xi_n}-\mathbb{E}(N_{\xi_n})|>(\varepsilon n-Cn^{2/3})\right)\leq \frac{\mathbb{V}(N_{\xi_n})}{(\varepsilon n-Cn^{2/3})^2}.
\end{eqnarray*}
Given the fact that $\mathbb{V}(N_{\xi_n})=O((1-\xi_n)^{-4})=O(n^{4/3})$, 
we have $1-\pi_{n,\varepsilon}\to 0$ as $n\to\infty$.

Next we prove the estimate of $\pi_n$. Note that 
$$\pi_n=M_n\cdot(\xi_n)^n/M(\xi_n)=P_n\cdot(\xi_n)^n/P(\xi_n).$$ 
Hence it is 
enough to obtain the asymptotics of $P_n$ and $(\xi_n)^n$    
as $n\to\infty$, and of $P(x)$  as $x\to 1^-$. 
These have first been found by Wright~\cite{Wr31}
and later by Meinardus in a more general framework~\cite{Meina} relying
on the saddle-point method (a detailed and accessible presentation of the saddle-point method is
given in~\cite[Ch.VIII]{FlaSe}, partitions are studied 
in the 6th section of the chapter). We briefly review the main ingredients. To find the asymptotics of $P(x)$ as $x\to 1^-$,
one applies the Mellin transform techniques to the series $L(t)=\ln(P(e^{-t}))$,
and finds 
\begin{equation}\label{eq:Pz}
P(x)\sim C_1(1-x)^{1/12}\exp\left( \zeta(3)\frac{x}{(1-x)^2}\right)\ \ \mathrm{as}\ x\to 1^-,
\end{equation}
with $C_1$ an explicit constant, $C_1\approx 0.9368$. 
Define $c:=(2\zeta(3))^{1/3}$, so $\zeta(3)=c^3/2$ and $\xi_n=1-cn^{-1/3}$.  
From~\eqref{eq:Pz} we obtain
$$
P(\xi_n)\sim C_1'n^{-1/36}\exp\big(\tfrac1{2}cn^{2/3}-\tfrac1{2}c^2n^{1/3} \big),
$$
with $C_1'=C_1c^{1/12}\approx 0.9599$.
The asymptotics of $(\xi_n)^n$ is easy to obtain. Since $\log(1-cn^{-1/3})=-cn^{-1/3}-\tfrac1{2}c^2n^{-2/3}-\tfrac{c^3}{3}n^{-1}+o(n^{-1})$, we obtain
$$
(\xi_n)^n=\exp(n\log(1-cn^{-1/3}))\sim C_2\exp\big(-cn^{2/3}-\tfrac1{2}c^2n^{1/3}\big),
$$
with $C_2=\exp(-c^3/3)\approx 0.4487$. Finally the asymptotics of $P_n$ has been obtained by Wright~\cite{Wr31} using
the saddle-point method and the estimate~\eqref{eq:Pz}. The idea
is to use Cauchy's formula
$$
P_n=\frac1{2i\pi}\int_{C(0,\xi_n)}P(z)z^{-n-1}\mathrm{d}z,
$$
with the circle of radius $\xi_n$ centered at $0$ as the integration contour.
Using the estimate~\eqref{eq:Pz} (more precisely one needs the fact that this estimate
holds in an open cone centered at $1$ and containing the line $z<1$), one 
shows that the main contribution of the integral is on a small arc of $C(0,\xi_n)$ around
the origin, and obtains
$$
P_n\mathop{\sim}_{n\to\infty}C_3n^{-25/36}\exp\big( \tfrac{3}{2}cn^{2/3}\big),
$$
 with $C_3\approx 0.2315$. 
 Thus, from the estimates of $P(\xi_n)$, $(\xi_n)^n$, and $P_n$,
 we find
 $$
 \pi_n\sim Cn^{-2/3},
 $$
 with $C\approx 0.1082$.
\end{proof}

It is easily checked that the expected running time of a rejection sampler is the expected running time
of the sampler times the expected number of calls (which is the inverse of the probability of success), therefore 
$$
\Lambda\Big(    \textsc{Sample}\cM[\xi_n;n]     \Big)=\frac{\Lambda M(\xi_n)}{\pi_n},\ \ \Lambda\Big(    \textsc{Sample}\cM[\xi_n;n,\varepsilon]     \Big)=\frac{\Lambda M(\xi_n)}{\pi_{n,\varepsilon}}.
$$ 
Since $\Lambda M(x)=O((1-x)^{-2})$ and $1-\xi_n=O(n^{-1/3})$,
we have $\Lambda M(\xi_n)=O(n^{2/3})$. Moreover $1/\pi_n=O(n^{2/3})$ and $1/\pi_{n,\varepsilon}\to 1$ (for fixed $\varepsilon\in(0,1)$) 
by Lemma~\ref{lem:prob}. Hence $\Lambda\big(\textsc{Sample}\cM(\xi_n;n)\big)=O(n^{4/3})$ and $\Lambda\big(\textsc{Sample}\cM(\xi_n;n,\epsilon)\big)=O(n^{2/3})$, which concludes the proof of Lemma~\ref{lem:targetM}. 

\subsection{Complexity of Pak's bijection}
The aim of this section is to provide an~$O$ bound on the expected running time
of Algorithm~\ref{algo:pak} for a multiset $\mu\in\cM$ of size $n$ taken 
uniformly at random. As we will see, the complexity of Algorithm~\ref{algo:pak}
applied to a multiset is expressed in terms of the width and length of the
bounding rectangle of $\mu$. These parameters have been recently studied
by Mutafchiev in~\cite{Mu06}: using the saddle-point method he shows 
that the width (similarly the length) after suitable
normalization converges weakly to an explicit distribution.
Since we are 
only interested in a big $O$, we only need upper bounds, which 
are much simpler to show. Therefore we prefer to provide here our own
simple self-contained analysis.

For a multiset $\mu\in\cM:=\MSet(\cZ\times\Seq(\cZ)^2)$ represented by its 
diagram, let~$w$ and $h$ be the width and height of 
the bounding rectangle of
 $\mu$. Pak's bijection scans the double range
 $[0\leq i\leq w-1,0\leq j\leq h-1]$; when a square $(i,j)$ is treated, the
 squares that are updated are those on the up-right diagonal
 $\{(i+c,j+c)\ \mathrm{such\ that}\ i+c\leq w,\ j+c\leq h\}$; 
each update of an entry consists
 of a fixed number of operations involving
 $\{+,-,\mathrm{max},\mathrm{min}\}$. The sum
 of the lengths of the up-right diagonals over the squares of the bounding 
rectangle is $\sum_{i=1}^{\mathrm{min}(w,h)}\!i(w-i+h-i+1)$, which is equal to
\begin{equation}\label{eq:sum}
\psi(w,h):=\tfrac{1}{2}L\ell(\ell+1)-\tfrac{1}{6}(\ell^3-1),\ \ \mathrm{where}\ L=\mathrm{max}(w,h),\ \ell=\mathrm{min}(w,h).
\end{equation}
This quantity is clearly $O(L^3)$, so that the complexity of 
 Algorithm~\ref{algo:pak} is cubic in $L$.

We introduce a parameter that will crucially simplify the analysis. 
Given $\mu\in\cM$ represented as a diagram, the \emph{hook-length} of an entry $(i,j)$
of the diagram is defined as $h(i,j):=i+j+1$, i.e., $h(i,j)$ is the size of $(i,j)$ seen as an element
of $\cA=\cZ\times\Seq(\cZ)^2$. The \emph{maximum hook-length} of $\mu$, denoted by $k(\mu)$,
is the maximum value of the hook-length over all non-zero entries of the diagram 
of $\mu$.

\begin{lemma}
\label{lem:compPak}
Given an element $\mu$ in $\cM:=\MSet(\cZ\times\Seq(\cZ)^2)$, the complexity of Pak's bijection
applied to $\mu$
is 
$O([k(\mu)]^3)$, where $k(\mu)$ is the maximal hook-length of $\mu$.
\end{lemma}
\begin{proof}
The maximal hook-length $k(\mu)$ is at least equal to the width $w$ and to the 
height $h$ of the bounding rectangle of $\mu$. Hence the complexity of Pak's bijection, which is $O([\mathrm{max}(w,h)]^3)$,  is also $O([k(\mu)]^3)$. 
\end{proof}

Hence, to have a big $O$ bound on 
the expected running time of Pak's bijection, we need to bound
the expected value of $k(\mu)^3$ for a multiset
 $\mu\in\cM$ of size $n$ taken uniformly at random. 
 First, for a 
series $C(x)=\sum_{n}c_nx^n$ with non-negative coefficients and with radius of convergence $\rho>0$, we recall
the trivial bound
\begin{equation}\label{eq:saddle_bound}
c_n\leq C(x)x^{-n}\ \ \mathrm{for\ any}\ x\in(0,\rho).
\end{equation}

\begin{lemma}[expected running time of Pak's bijection at a fixed size]\label{lem:pakExpe}
For $n\geq 1$, let $\mu\in\cM$ be a multiset of size $n$ taken uniformly at random.
Then the expected running time of Algorithm~\ref{algo:pak} applied to $\mu$
is $O(n(\ln n)^3)$. 
\end{lemma}
\begin{proof}
Denote by $H_n$ the expectation of $k(\mu)^3$ for a multiset $\mu\in\cM$ of size
$n$ taken uniformly at random. By Lemma~\ref{lem:compPak}, 
the expected running time of Algorithm~\ref{algo:pak} under the uniform distribution (on $\cM$) at size $n$ is $O(H_n)$.
Hence to show the lemma we just have to show that $H_n=O(n(\ln n)^3)$. 
For $k\geq 1$, denote by $\cM^{(k)}$ 
 the family of multisets in $\cM$ with maximal hook-length equal to $k$, and denote by $M^{(k)}(x)$  the series of $\cM^{(k)}$. Define
$$
K(x):=\sum_{k\geq 1}k^3M^{(k)}(x).
$$
Note that $H_n=[x^n]K(x)/[x^n]M(x)=[x^n]K(x)/P_n$, with $P_n$
the number of plane partitions of size $n$. 
For $k\geq 1$, define a \emph{$k$-pointed multiset} as a multiset $\mu\in\cM$ where 
 a non-zero entry at hook-length $k$ is marked in the diagram of $\mu$. 
 Note that the series of $k$-pointed multisets is $kx^kM(x)$, where the factor $k$
 counts the possible places to mark an entry and where the factor $x^k$ takes 
 account of the fact that the marked entry is non-zero. Since $k$-pointed multisets form
 a superfamily of $\cM^{(k)}$, we have
 $$
 M^{(k)}(x)\leq kx^kM(x)=kx^kP(x),\ \ \mathrm{for\ any\ }0<x<1.
 $$ 
Let $B$ be a constant whose value is to be fixed later, and define $u_n:=\lfloor Bn^{1/3}\log(n)\rfloor$. 
We have
$$
H_n=\frac1{P_n}\sum_{k=1}^nk^3[x^n]M^{(k)}(x)\leq (u_n)^3+\frac{n^3}{P_n}\sum_{k=u_n}^n[x^n]M^{(k)}(x).
$$
As $M^{(k)}(x)\leq kx^kP(x)$, the bound~\eqref{eq:saddle_bound} ensures that, for $u_n\leq k\leq n$, 
$$
[x^n]M^{(k)}(x)\leq M^{(k)}(\xi_n)\cdot(\xi_n)^{-n} \leq k(\xi_n)^{k}\cdot P(\xi_n)\cdot (\xi_n)^{-n}\leq n(\xi_n)^{u_n}\cdot P(\xi_n)\cdot (\xi_n)^{-n}.
$$
Hence
$$
H_n\leq (u_n)^3+n^5(\xi_n)^{u_n}\frac{P(\xi_n)}{P_n\cdot(\xi_n)^n}.
$$
Let $c:=(2\zeta(3))^{1/3}$. We have
$$
(\xi_n)^{u_n}=\exp\Big(u_n\log(1-cn^{-1/3})\Big)\sim\exp(-cu_nn^{-1/3})\sim n^{-cB}.
$$
Moreover, according to Lemma~\ref{lem:prob}, 
 $$
\frac{P(\xi_n)}{P_n\cdot (\xi_n)^{n}}=\frac{1}{\pi_n}=O(n^{2/3}).
 $$
Hence
$$
H_n=O(n(\ln n)^3)+O(n^5n^{-cB}n^{2/3}).
$$
Taking the constant $B$ sufficiently large so that $cB>4+2/3$ (e.g., $B=5$), we obtain
$$
H_n=O(n(\ln n)^3).
$$
\end{proof}

\begin{proposition}
\label{prop:compl}
For any $\varepsilon\in(0,1)$, 
the expected running time of $\textsc{SamplePartitions}[n,\varepsilon]$ 
satisfies
$$
\Lambda\Big(\textsc{SamplePartitions}[n,\varepsilon]\Big)=O(n(\ln n)^3)\ \ \ \ \mathrm{as}\ n\to\infty,
$$
the asymptotic constant in the big $O$ being independent of $\varepsilon$.

The expected running time of $\textsc{SamplePartitions}[n]$ satisfies
$$
\Lambda\Big(\textsc{SamplePartitions}[n]\Big)=O(n^{4/3})\ \ \ \ \mathrm{as}\ n\to\infty.
$$
\end{proposition}
\begin{proof}
Start with the proof for the exact-size sampler. 
By definition, we have 
$$
\Lambda\Big(\textsc{SamplePartitions}[n]\Big)=\Lambda\Big(\textsc{Sample}\cM[\xi_n;n]\Big)+\mathbb{E}_n(\textsc{PakBijection}),
$$
where $\mathbb{E}_n(\textsc{PakBijection})$ is the expected running time
of Algorithm~\ref{algo:pak} (Pak's bijection) 
for a multiset of size $n$ taken uniformly at random. 
Lemma~\ref{lem:targetM} ensures that 
$\Lambda(\textsc{Sample}\cM[\xi_n;n])$ is $O(n^{4/3})$;
by Lemma~\ref{lem:pakExpe}, $\mathbb{E}_n(\textsc{PakBijection}$ is $O(n(\ln n)^3)$.
Hence   $\Lambda(\textsc{SamplePartitions}[n])$ is $O(n^{4/3})$. 

Consider now the approximate-size sampler. By definition, we have
$$
\Lambda\Big(\textsc{SamplePartitions}[n,\varepsilon]\Big)=
\Lambda\Big(\textsc{Sample}\cM[\xi_n;n,\varepsilon]\Big)+\mathbb{E}_{n,\varepsilon}(\textsc{PakBijection}),
$$
where $\mathbb{E}_{n,\varepsilon}(\textsc{PakBijection})$ is the expected running time
of Pak's bijection for a multiset drawn from $\textsc{Sample}\cM(\xi_n;n,\varepsilon)$. 
By Lemma~\ref{lem:targetM}, 
$\Lambda(\textsc{Sample}\cM(\xi_n;n,\varepsilon))$ is $O(n^{2/3})$.
Since the size of an object output by $\textsc{Sample}\cM(\xi_n;n,\varepsilon)$
is at most $2n$ (because $\varepsilon\in(0,1)$),   
Lemma~\ref{lem:pakExpe}   
ensures that $\mathbb{E}_{n,\varepsilon}(\textsc{PakBijection})$ is $O(n(\ln n)^3)$.
Hence   $\Lambda(\textsc{SamplePartitions}[n,\epsilon])$ is $O(n(\ln n)^3)$. 
\end{proof}


\subsection{Complexity of the samplers for $(a\times b)$-boxed plane 
partitions}\label{sec:abboxed}

By definition (see Algorithm~\ref{algo:genMpq}), 
the Boltzmann sampler $\Gamma M_{a,b}(x)$
for $(a\times b)$-boxed multisets just consists of $ab$ calls to geometric laws.
Giving unit cost to a call to a geometric law, one has
$$
\Lambda M_{a,b}(x)=ab.
$$
Next, recall that 
$$
M_{a,b}(x)=\prod_{\substack{0\leq i<a\\0\leq
j<b}}\frac{x^{i+j+1}}{1-x^{i+j+1}}\mathop{\sim}_{x\to
1^-}\frac{c}{(1-x)^{ab}},\ \mathrm{with}\ c=\prod_{\substack{0\leq i<a\\0\leq
j<b}}\frac{1}{i+j+1}.
$$
For $\alpha>0$, a class $\cC$ is called $\alpha$-singular at $x=1$ if $C(x)\sim c/(1-x)^{\alpha}$ for some constant $c>0$ 
(the $\sim$ holding in a complex neighbourhood of $1$) and if $1$ is 
the only singularity of $C(x)$ in a disk of the form $\{z\in\mathbb{C}\ \mathrm{s.t.}\ |z|<1+\delta\}$ for some $\delta>0$.
Note that $\cM_{a,b}$ is $\alpha$-singular for $\alpha=ab$.
 In~\cite{DuFlLoSc04} 
it is shown that if $\cC$ is $\alpha$-singular
at $x=1$, then for $n\geq 1$ the probability $\pi_n$ of being of size $n$ under
the Boltzmann model at $x_n:=1-\alpha/n$ satisfies
\begin{equation}\label{eq:pinalpha}
\pi_n\mathop{\sim}_{n\to\infty}\frac{\phi(\alpha)}{n},\ \ \mathrm{where}\ \phi(\alpha)=\frac{(\alpha/e)^{\alpha}}{\Gamma(\alpha)}. 
\end{equation}
And, for fixed $\varepsilon\in(0,1)$, the probability $\pi_{n,\varepsilon}$ 
of being in the size-domain $[n(1-\varepsilon),n(1+\varepsilon)]$ under the Boltzmann 
model at $x_n$ satisfies
\begin{equation}\label{eq:pinepsalpha}
\pi_{n,\varepsilon}\mathop{\to}_{n\to\infty}\Phi(\alpha,\varepsilon),\ \ \mathrm{where}\ \Phi(\alpha,\varepsilon)=\frac{(\alpha/e)^{\alpha}}{\Gamma(\alpha)}\int_{-\varepsilon}^{\varepsilon}(1+s)^{\alpha-1}e^{-\alpha s}\mathrm{d}s.
\end{equation}
Since the expected running time of a rejection sampler is the expected running time of the sampler divided
by the probability of success at each attempt, the expected running times of the targeted samplers
for $\cM_{a,b}$ satisfy asymptotically 
$$
\Lambda\Big(\textsc{Sample}\cM_{a,b}[\xi_n^{a,b};n]\Big)\mathop{\sim}_{n\to\infty}\frac{ab}{\phi(ab)}n,\ \ \Lambda\Big(\textsc{Sample}\cM_{a,b}[\xi_n^{a,b};n,\varepsilon]\Big)\mathop{\to}_{n\to\infty}\frac{ab}{\Phi(ab,\varepsilon)}.
$$

The second step of the targeted samplers for $(a\times b)$-boxed plane partitions
is Algorithm~\ref{algo:pak} (Pak's bijection). 
When $x\to 1^-$, all entries of the rectangle $R_{a,b}:=[0..a-1]\times[0..b-1]$ 
in the diagram of 
$\mu\leftarrow\Gamma M_{a,b}(x)$ are non-zero
with high probability, hence 
the bounding rectangle of $\mu$  is $R_{a,b}$ with high probability. As 
a consequence, the complexity of Algorithm~\ref{algo:pak} applied to $\mu$
is with high probability the quantity $\psi(a,b)$ defined in~\eqref{eq:sum}. 
Hence
$$
\Lambda\Big(\textsc{SamplePartitions}_{a,b}[n]\Big)\mathop{\sim}_{n\to\infty}\frac{ab}{\phi(ab)}n+\psi(a,b)\mathop{\sim}_{n\to\infty}\frac{ab}{\phi(ab)}n,$$
$$\Lambda\Big(\textsc{SamplePartitions}_{a,b}[n,\varepsilon]\Big)\mathop{\to}_{n\to\infty}\frac{ab}{\Phi(ab,\varepsilon)}+\psi(a,b),
$$
which concludes the proof of Theorem~\ref{theo:box_plane}. 

\vspace{0.5cm}

\noindent\emph{Acknowledgement.} The authors would like 
to thank Philippe Flajolet for his help 
to analyse the algorithm. 
The article has also greatly benefited from a
 discussion with Christian Krattenthaler and from detailed comments
 of Mireille Bousquet-M\'elou and of an anonymous referee
on a first manuscript.   
We finally thank Gu\'ena\"el Renault for the 3D drawings of plane partitions.
\bibliographystyle{plain}
\bibliography{mabiblio}

\begin{thebibliography}{10}

\bibitem{BeEdKn}
E.~A. Bender and D.~E. Knuth.
\newblock Enumeration of plane partitions.
\newblock {\em J. Combin. Theory Ser. A}, 13(1):40--54, 1972.

\bibitem{Bre99}
David~M. Bressoud.
\newblock {\em Proofs and confirmations: the story of the alternating sign
  matrix conjecture}.
\newblock Cambridge University Press, New York, NY, USA, 1999.

\bibitem{CeKe}
R.~Cerf and R.~Kenyon.
\newblock The low-temperature expansion of the {W}ulff crystal in the 3{D}
  {I}sing model.
\newblock {\em Comm. Math. Phys.}, 222(1):147--179, 2001.

\bibitem{CoLaPr98}
H.~Cohn, M.~Larsen, and J.~Propp.
\newblock The shape of a typical boxed plane partition.
\newblock {\em New York J. Math.}, 4:137--166, 1998.

\bibitem{DuFlLoSc04}
P.~Duchon, P.~Flajolet, G.~Louchard, and G.~Schaeffer.
\newblock Boltzmann samplers for the random generation of combinatorial
  structures.
\newblock {\em Combin. Probab. Comput.}, 13(4--5):577--625, 2004.
\newblock Special issue on Analysis of Algorithms.

\bibitem{FlFuPi07}
P.~Flajolet, \'E. Fusy, and C.~Pivoteau.
\newblock Boltzmann sampling of unlabelled structures.
\newblock In {\em Proceedings of the 4th Workshop on Analytic Algorithms and
  Combinatorics, ANALCO'07 (New Orleans)}, pages 201--211. SIAM, 2007.

\bibitem{FlGoDu95}
P.~Flajolet, X.~Gourdon, and P.~Dumas.
\newblock Mellin transforms and asymptotics: Harmonic sums.
\newblock {\em Theoret. Comput. Sci.}, 144(1--2):3--58, June 1995.

\bibitem{FlaSe}
P.~Flajolet and R.~Sedgewick.
\newblock {\em Analytic combinatorics}.
\newblock Cambridge University Press, 2009.

\bibitem{FlZiVa94}
P.~Flajolet, P.~Zimmerman, and B.~Van Cutsem.
\newblock A calculus for the random generation of labelled combinatorial
  structures.
\newblock {\em Theoret. Comput. Sci.}, 132(1-2):1--35, 1994.

\bibitem{Fu07a}
\'E. Fusy.
\newblock Uniform random sampling of planar graphs in linear time, 2007.
\newblock arXiv:0705.1287, to appear in Random Structures Algorithms.

\bibitem{Ga81}
E.~R. Gansner.
\newblock Matrix correspondences of plane partitions., 1981.

\bibitem{Kn70}
D.~E. Knuth.
\newblock Permutations, matrices, and generalized {Young} tableaux.
\newblock {\em Pacific J. Math.}, 34:709--727, 1970.

\bibitem{Kra99}
C.~Krattenthaler.
\newblock Another involution principle-free bijective proof of {S}tanley's
  hook-content formula.
\newblock {\em J. Combin. Theory Ser. A}, 88(1):66--92, 1999.

\bibitem{MaJa05}
J.~Ma and E.~J. Janse~van Rensburg.
\newblock Rectangular vesicles in three dimensions.
\newblock {\em J. Phys. A}, 38(19):4115--4147, 2005.

\bibitem{Mc12}
P.~A. MacMahon.
\newblock Memoir on the theory of the partitions of numbers. vi: Partitions in
  two-dimensional space, to which is added an adumbration of the theory of
  partitions in three-dimensional space.
\newblock {\em Phil. Trans. Roy. Soc. London Ser. A}, 211:345--373, 1912.

\bibitem{MaNa}
T.~Maeda and T.~Nakatsu.
\newblock Amoebas and instantons.
\newblock {\em Int. J. Mod. Phys. A}, 22:937--984, 2007.

\bibitem{Meina}
G.~Meinardus.
\newblock Asymptotische {A}ussagen \"uber {P}artitionen.
\newblock {\em Math. Z.}, 59:388--398, 1954.

\bibitem{Mu06}
L.~Mutafchiev.
\newblock The size of the largest part of random plane partitions of large
  integers.
\newblock {\em Integers}, 6:A13, 2006.

\bibitem{Mu}
L.~Mutafchiev and E.~Kamenov.
\newblock Asymptotic {F}ormula for the {N}umber of {P}lane {P}artitions of
  {P}ositive {I}ntegers.
\newblock {\em C. R. Acad. Bulgare Sci.}, 59(4):361--366, 2006.

\bibitem{NiWi78}
A.~Nijenhuis and H.~S. Wilf.
\newblock {\em Combinatorial Algorithms}.
\newblock Academic Press, second edition, 1978.

\bibitem{NoPaSt97}
J.-C. Novelli, I.~Pak, and A.~V. Stoyanovskii.
\newblock A direct bijective proof of the hook-length formula.
\newblock {\em Discrete Math. Theor. Comput. Sci.}, 1(1):53--67, 1997.

\bibitem{okounkov-2003}
A.~Okounkov and N.~Reshetikhin.
\newblock Correlation function of schur process with application to local
  geometry of a random 3-dimensional young diagram.
\newblock {\em J. Amer. Math. Soc.}, 16(3):581--603, 2003.

\bibitem{okounkov-2005}
A.~Okounkov and N.~Reshetikhin.
\newblock Random skew plane partitions and the pearcey process.
\newblock {\em Comm. Math. Phys.}, 269(3), February 2007.

\bibitem{Pak02}
I.~Pak.
\newblock Hook length formula and geometric combinatorics.
\newblock {\em S\'eminaire Lotharingien de Combinatoire}, 46:Art. B46f, 13 pp.
  (electronic), 2001/02.

\bibitem{PiSaSo07}
C.~Pivoteau, B.~Salvy, and M.~Soria.
\newblock Boltzmann oracle for combinatorial systems.
\newblock In {\em Algorithms, Trees, Combinatorics and Probabilities}, pages
  475--488. Discrete Mathematics and Theoretical Computer Science, 2008.
\newblock Proceedings of the Fifth Colloquium on Mathematics and Computer
  Science. Blaubeuren, Germany. September 22-26, 2008.

\bibitem{PrGen}
J.~Propp.
\newblock Generating random elements of finite distributive lattices.
\newblock {\em Electron. J. Combin.}, 4(2), 1997.
\newblock R15, 12p.

\bibitem{Ve96}
A.~M. Vershik.
\newblock Statistical mechanics of combinatorial partitions, and their limit
  configurations.
\newblock {\em Funct. Anal. Appl.}, 30(2):90--105, 1996.

\bibitem{Wil97}
D.~B. Wilson.
\newblock Determinant algorithms for random planar structures.
\newblock In {\em SODA '97: Proceedings of the eighth annual ACM-SIAM symposium
  on Discrete algorithms}, pages 258--267, Philadelphia, PA, USA, 1997. Society
  for Industrial and Applied Mathematics.

\bibitem{Wr31}
E.~M. Wright.
\newblock Asymptotic partition formulae, i: Plane partitions.
\newblock {\em Quart. J. Math. Oxford}, Ser. 2:177--189, 1931.

\bibitem{Yo01}
A.~Young.
\newblock On quantitative substitutional analysis.
\newblock {\em Proc. Lond. Math. Soc.}, 33:97--146, 1901.

\bibitem{Ze96}
D.~Zeilberger.
\newblock Proof of the alternating sign matrix conjecture.
\newblock {\em Electron. J. Combin.}, 3(2):Research Paper 13, approx.\ 84 pp.\
  (electronic), 1996.
\newblock The Foata Festschrift.

\end{thebibliography}

\end{document}